\documentclass{amsart}
\usepackage{graphicx} 
\usepackage{amssymb}
\usepackage{amsthm}
\usepackage{mathtools}
\usepackage[dvipsnames]{xcolor}
\usepackage{pgfplots}
\usepackage{color}
\usepackage{hyperref}
\usepackage{tikz}
\usepackage{xfrac}
\usepackage{cite}
\usepackage[foot]{amsaddr}
\usetikzlibrary{positioning}
\usetikzlibrary{arrows.meta,decorations.pathreplacing}

\hypersetup{
	colorlinks=true,       
	linkcolor=blue,          
	citecolor=blue,        
	filecolor=blue,      
	urlcolor=blue           
}

\definecolor{myblue}{RGB}{0 0 255}
\definecolor{myred}{RGB}{255 0 0}

\newcommand{\occ}[4]{%
	\ifnum#3=2
	\filldraw[fill=#4!30, draw=#4] (#1,#2) circle (0.19);
	\node[text=#4, font=\scriptsize\bfseries] at (#1,#2) {2};
	\else
	\filldraw[fill=#4!30, draw=#4] (#1,#2) circle (0.19);
	\node[text=#4, font=\scriptsize\bfseries] at (#1,#2) {1};
	\fi}

\newcommand{\scaffold}[1]{%
	\foreach \k in {-2,-1,0,1,2}{\draw[black!22, thick] (\k,0.25) -- (\k,-6.1);}
	\ifnum#1=1
	\draw[dashed, black] (-2.4,-2.5) -- (2.4,-2.5);
	\fi}

\newcommand{\cont}[2]{%
	\draw[#2, thick, -{Stealth[length=4pt]}, shorten <=6pt] (#1,-5)--(#1,-5.55);
	\node[#2, font=\Huge] at (#1,-5.92) {$\vdots$};}

\newcommand{\spin}[3]{%
	\filldraw[fill=#3!30, draw=#3, line width=0.4pt]
	(#1-0.5,#2-0.5) rectangle (#1+0.5,#2+0.5);}

\newcommand{\bh}[3]{\draw[#3, line width=2pt] (#1-0.5,#2)--(#1+0.5,#2);}
\newcommand{\bv}[3]{\draw[#3, line width=2pt] (#1,#2-0.5)--(#1,#2+0.5);}

\newcommand{\gridbg}{%
	\foreach \a in {-5.5,-4.5,...,5.5}{%
		\draw[black!30, thin] (\a,-5.5)--(\a,5.5);
		\draw[black!18, thin] (-5.5,\a)--(5.5,\a);}}

\newcommand{\boxmark}{\draw[dashed, black, line width=0.9pt]
	(-2.5,-2.5) rectangle (2.5,2.5);}

\newcommand{\bcframe}{\draw[black!70, line width=2pt] (-5.5,-5.5) rectangle (5.5,5.5);}

\newcommand{\axes}{%
  \fill[black!10] (2,-1.5) rectangle (6.35,1.5);                       
  \draw[-{Stealth[length=4pt]}, black] (0,0) -- (6.45,0);          
  \draw[{Stealth[length=4pt]}-{Stealth[length=4pt]}, black] (0,-1.5) -- (0,1.5); 
  \node[black, font=\scriptsize] at (-0.2,1.35) {$\mathbb{R}$};
  \node[black, font=\scriptsize] at (-0.2,-0.26) {$0$};
  \draw[dashed, black, line width=0.8pt] (2,-1.5) -- (2,1.5);}
  
\numberwithin{equation}{section}
\numberwithin{figure}{section}

\theoremstyle{plain} \newtheorem{theorem}{Theorem}[section]
\theoremstyle{plain} 
\theoremstyle{plain} \newtheorem{lemma}[theorem]{Lemma}
\theoremstyle{plain} 
\theoremstyle{definition} 
\theoremstyle{definition} 
\theoremstyle{remark} \newtheorem{remark}[theorem]{Remark}
\theoremstyle{remark} \newtheorem{example}[theorem]{Example}

\renewcommand{\P}{\mathbb P}

\newcommand{\R}{\mathbb R}

\newcommand{\mbb}{\mathbb}
\newcommand{\Rbb}{\mathbb{R}}
\newcommand{\Borel}{\mathcal{B}}

\newcommand{\dualInd}{\mathcal{I}}
\newcommand{\dualFunc}{\mathcal{F}}

\newcommand{\diff}{\,\textnormal{d}}
\newcommand{\ind}{\mathbf{1}}

\title[Coupling Theory, Optimal Transport, and Strassen's Theorem]{Coupling Theory, Optimal Transport, and Strassen's Theorem for Eventual Domination}

\author{Adam Quinn Jaffe}
\address[AQJ]{Department of Statistics, Columbia University, New York, USA}

\author{Daniel Raban}
\address[DR]{Department of Statistics, University of California Berkeley, Berkeley, CA, USA}

\email{a.q.jaffe@columbia.edu}
\email{danielraban@berkeley.edu}

\thanks{This material is based upon work for which DR was supported by a research contract with Sandia National Laboratories, a U.S.\ Department of Energy multimission
laboratory.}

\subjclass[2020]{60A10, 28A35, 90C46}
	
\keywords{equivalence relations,
Kantorovich duality, partial orders,
stochastic domination, Strassen's theorem}
	
\date{\today}

\begin{document}

\begin{abstract}
Many results in probability (most famously, Strassen's theorem on stochastic domination), characterize some relationship between probability distributions in terms of the existence of a particular structured coupling between them.
    Optimal transport, and in particular Kantorovich duality, provides a common framework for these results.
    We use this perspective to study ``eventual domination'', a class of orders arising naturally in stochastic processes (including in the analysis of branching random walks, Ising models, and diffusions), which does not satisfy the topological conditions required by standard optimal transport theory.
    More generally, we study the connection between distributional relations and their coupling counterparts for topologically irregular preorders (e.g. equivalence relations, partial orders). \\
    \indent To this end, we show that Strassen’s theorem “nearly holds” in this topologically irregular setting but that the full theorem admits counterexamples, including for eventual domination.
    The core of the proof is a novel technical result in optimal transport, which shows that the Kantorovich dual problem is well-behaved for cost functions which can be written as a non-increasing limit of lower semi-continuous functions.
\end{abstract}

\maketitle


\section{Introduction}

\subsection{Background}

A \textit{coupling} of Borel probability measures $\mu,\nu$ on a Polish space $S$ is a  Borel probability measure on $S\times S$ satisfying $\pi(S\times \,\cdot\,) = \mu$ and $\pi(\,\cdot\,\times S) = \nu$.
Equivalently, a coupling is the joint law of a pair of $S$-valued random variables $(X,Y)$ defined on a common probability space in such a way that the marginal distributions of $X$,$Y$ are $\mu$,$\nu$, respectively.
Coupling is one of the core techniques in probability theory, where it is used ubiquitously in essentially all sub-disciplines; see \cite{denHollander, ThorissonII, LindvallBook} for introductions.
At the same time, the study of couplings is interesting in its own right and a great deal literature has addressed fundamental questions in coupling theory, particularly in the setting of stochastic processes \cite{Griffeath, Pitman, Goldstein, SverchkovSmirnov, MEXIT, HsuSturm, SchillingLevy1, BurdzyKendall, BenArousCranstonKendall, BanerjeeKendall2, Kuwada, HummelJaffe} but also more generally \cite{Georgii, Shinko, JaffeCoupling}.

The richness of the theory and applications of couplings led Thorisson to state the ``working hypothesis'' in \cite{ThorissonI} that ``each meaningful distributional relation should have a coupling counterpart''.
Well-known examples include Strassen's theorem on stochastic domination \cite{Lindvall}, characterizations of coalescence of stochastic processes \cite{Griffeath, Pitman, Goldstein}, characterization of Palm measures of stationary point processes \cite{LastThorisson}, Strassen's theorem on convex order \cite{Strassen}, and more.
The working hypothesis can also be seen as a practical matter, as it allows one to develop methods for testing partially-identified statistical models wherein the joint distribution of two random variables is unobserved while their marginal distributions are observed \cite{GalichonHenry, EkelandGalichonHenry,YangYang,CoteWang}.
The goal of this paper, plainly, is to better understand the extent to which Thorisson's working hypothesis is true.

As a concrete example, consider Strassen's theorem on stochastic domination \cite{Strassen, Lindvall}.
It states that if $\le$ is a partial order on $S$ such that the set $R_{\le} := \{(x,y)\in S\times S: x\le y\}$ is sufficiently regular then the following are equivalent:
\begin{itemize}
    \item[(i)] We have $\mu(A) \le \nu(A)$ for all monotone Borel sets $A\subseteq S$.
    \item[(ii)] There exists a coupling $\pi$ of $\mu$ and $\nu$ satisfying $\pi(R_{\le}) = 1$.
\end{itemize}
Here, a Borel set $A\subseteq S$ is called \textit{monotone} if $x\in A$ and $x\le y$ imply $y\in A$.
In terms of the working hypothesis, statement (i) is a distributional relation and statement (ii) is its coupling counterpart.

Another example, called simply \textit{strong duality} in existing literature, comes from the recent works \cite{JaffeCoupling, RigoPratelli}, which state that if $\cong$ is an equivalence relation on $S$ such that the set  $R_{\cong} := \{(x,y)\in S\times S: x\cong y\}$ is regular enough, then the following are equivalent:
\begin{itemize}
    \item[(i)] We have $\mu(A) = \nu(A)$ for all invariant Borel sets $A\subseteq S$.
    \item[(ii)] There exists a coupling $\pi$ of $\mu$ and $\nu$ satisfying $\pi(R_{\cong}) = 1$.
\end{itemize}
Here, a Borel set $A\subseteq S$ is called \textit{invariant} if $x\cong y$ implies that we have $x\in A$ if and only if $y\in A$.
As before, we cast this in terms of the working hypothesis by noting that statement (i) is a distributional relation and statement (ii) is its coupling counterpart.

In the setting of stochastic processes, we are particularly interested in distributional relations and their coupling counterparts for probability measures on spaces of paths.
Specifically, we study the relation of ``eventual domination'' whereby $S$ is endowed with a partial order $\le$ and we relate $(x,y)$ if and only if $x_n \le y_n$ for all sufficiently large $n\in\mathbb{N}$.
We discuss eventual domination in detail in Section~\ref{sec:examples} where we also give several examples and applications from the study of branching random walks, Ising models, and diffusions.

Similar relations have been extensively studied in other parts of probability theory.
For instance, let $\Sigma$ be a Polish space and consider the path space $S=\Sigma^{\mathbb{N}}$; write $x=\{x_n\}_{n\in\mathbb{N}}$ for a generic path in $S$.
If $\Sigma$ is equipped with a partial order $\le$, then many works (e.g., \cite{Massey, RdorfMCs, LopezMartinezSanz, LevyCoupling, BrandtLast, Leskela}) have studied couplings satisfying ``pointwise domination'', i.e., the coupling statement (ii) for the partial order on $S$ which relates $(x,y)$ if and only if $x_n\le y_n$ for all $n\in\mathbb{N}$.
Or, if $\Sigma$ is a general Polish space then many works (e.g., \cite{HsuSturm, SchillingLevy1, BurdzyKendall, BenArousCranstonKendall, BanerjeeKendall2, Kuwada}) have studied couplings satisfying ``eventual equality'', i.e., the coupling problem (ii) which relates $(x,y)$ if and only if $x_n= y_n$ for all sufficiently large $n\in\mathbb{N}$.

\subsection{Problem Statement}

Presently, we propose that Thorisson's working hypothesis should be understood as a statement about Kantorovich duality for optimal transport problems whose cost functions take only values zero and one.
For a formal explanation of this, suppose that $\mu$ and $\nu$ are Borel probability measures on $S$, and that  $R\subseteq S\times S$ is a general measurable relation.
The question of whether there exists a coupling $\pi$ of $\mu$ and $\nu$ satisfying $\pi(R) = 1$ leads one to naturally consider the optimal transport problem
\begin{equation}\label{eqn:cpl-primal}
    \begin{cases}
        \textnormal{minimize} &\int_{S\times S}(1-\ind_{R})\diff \pi \\
        \textnormal{over} &\textnormal{couplings } \pi \textnormal{ of } \mu,\nu.
    \end{cases}
\end{equation}
At the same time, general results on Kantorovich duality show that \eqref{eqn:cpl-primal} is equivalent (in a suitable sense) to the problem
\begin{equation}\label{eqn:cpl-dual}
    \begin{cases}
        \textnormal{maximize} &\int_{S}\phi\diff \mu + \int_{S}\psi \diff \nu \\
        \textnormal{over} &\textnormal{bounded measurable }\phi,\psi:S\to\Rbb \\
        \textnormal{with} &\phi(x)+\psi(y) \le 1-\ind_{R}(x,y)\textnormal{ for all }x,y\in S.
    \end{cases}
\end{equation}
Thus, by considering the distributional statement (i) that problem \eqref{eqn:cpl-dual} has optimal value zero, and the coupling counterpart (ii) that  problem \eqref{eqn:cpl-primal} has optimal value zero, we have realized Thorisson's working hypothesis as a statement about the duality between \eqref{eqn:cpl-dual} and \eqref{eqn:cpl-primal}.

Towards making a rigorous proof out of this formal argument, we consider the setting above and we apply the result \cite[Corollary~2.3.9]{RachevRdorf} to get
\begin{equation}\label{eqn:Kantorovich-general}
    \sup_{\phi,\psi}\left(\int_{S}\phi\diff \mu + \int_{S}\psi \diff \nu\right) = \inf_{\pi} \int (1-\ind_{R})\diff \pi,
\end{equation}
where the domains of the supremum and infimum are specified above.
In other words, problems~\eqref{eqn:cpl-primal} and~\eqref{eqn:cpl-dual} always have the same value.
Completing the argument, then, reduces to answering the following questions:
\begin{itemize}
    \item When is problem~\eqref{eqn:cpl-dual} ``meaningful''?    
    \item When does problem~\eqref{eqn:cpl-primal} admit minimizers?
\end{itemize}
One way of interpreting the term ``meaningful'' is that we require some condition on the (marginal) probabilities assigned by $\mu,\nu$ to some events $A$ in a specified collection $\mathcal{A}$ of Borel sets.
In other words, we would like to compare $\mu,\nu$ by the values they assign to the \textit{same} event $A\in\mathcal{A}$, meaning the Kantorovich problem \eqref{eqn:cpl-dual} should have one dual variable rather than two, and that the dual variables should be indicator functions rather than general measurable functions.

In previous works, affirmative answers to the above questions generally rely on strict topological assumptions about $R$.
For instance, for a partial order $\le$, a standard assumption for Strassen's theorem to hold (see \cite{Strassen}) is that $R_{\le}$ is a closed subset of $S\times S$.
Analogously, for an equivalence relation $\cong$, a sufficient condition for strong duality to hold (see \cite{JaffeCoupling}) is that $R_{\cong}$ is a closed subset of $S\times S$.
In fact, the result of \cite{Lindvall} shows that if $R\subseteq S\times S$ is any closed set corresponding to a reflexive transitive relation (sometimes called \textit{pre-order}), then a form of Strassen's theorem holds.
In this sense, we may regard Thorisson's working hypothesis as completely settled for closed sets corresponding to reflexive transitive relations, but this does not include the pre-order $\preceq$ of eventual equality.
We thus aim to develop stronger results anticipated by the recent works \cite{JaffeCoupling,RigoPratelli} which show that, for equivalence relation $\cong$, one can replace closedness of $R_{\cong}\subseteq S\times S$ with much broader measure-theoretic conditions.

Let us briefly digress to make some basic observations about the connection between the working hypothesis and optimal transport, namely that many of the usual objects of interest in abstract optimal transport (see\cite{Villani}) take on concrete and interpretable forms when the cost function is $c=1-\ind_{R}$ for a suitable relation $R$ on $S$.
For example, suppose that $\preceq$ is a partial order and that $R=R_{\preceq}$.
Then, a function $\phi:S\to\Rbb$ is $c$-convex if and only if it is (non-negative and) monotone, and the $c$-transform $\phi^c$ of $\phi$ is (the negative of) its smallest monotone majorant.
Also, a Borel set $A\subseteq S$ is monotone if and only if $\ind_{A}:S\to\Rbb$ is a $c$-convex function, and the monotone closure $A'$ of $A$ is characterized via $\ind_{A}^c = -\ind_{A'}$.
For another example, suppose that $\cong$ is an equivalence relation and that $R=R_{\cong}$.
Then, a function $\phi$ is $c$-convex if and only if it is (non-negative and) invariant, and the $c$-transform $\phi^c$ is (the negative of) its smallest invariant majorant.
Also, a Borel set $A\subseteq S$ is invariant if and only if $\ind_{A}$ is a $c$-convex function, and the saturation $A'$ of $A$ is characterized via $\ind_{A}^c = -\ind_{A'}$.
Some of these observations have been made implicitly in the work \cite{Rdorf1}, which provides a proof of Strassen's theorem via optimal transport for partial orders $\preceq$ such that $R_{\preceq}$ is closed in $S\times S$.

Returning to the present work, our goal is investigate Thorisson's working hypothesis for reflexive, transitive relations that need not be closed sets, but can be written as a countable increasing union of closed sets; that is, we focus on reflexive, transitive, $F_{\sigma}$ relations.
As we have discussed, our main motivations for this level of generality is to study the pre-order of eventual domination, which is reflexive, transitive, and $F_{\sigma}$, but not antisymmetric.
Additionally, we are motivated by prior literature \cite{JaffeCoupling, RigoPratelli} studying equivalence relations $\cong$, which show that strong duality holds under some purely measure-theoretic conditions on $\cong$;
notably, these works show that strong duality holds in some settings where standard approaches based on Kantorovich duality \cite{Rdorf1} can fail.

\subsection{Main Results}

Now we can state the results of the paper.
Throughout, we use the notation that $\mathcal{B}(S)$ denotes the collection of Borel sets of $S$, and $\Pi(\mu,\nu)$ denotes the space of all couplings of Borel probability measures $\mu,\nu$.
Our main result is the following, which states that \eqref{eqn:cpl-dual} is ``meaningful'' but that \eqref{eqn:cpl-primal} need not admit solutions:

\begin{theorem}\label{thm:main}
    If $S$ is a Polish space with its Borel $\sigma$-algebra $\mathcal B(S)$, and $R \subseteq S \times S$ is a reflexive, transitive, $F_\sigma$ relation, then we have
    \begin{equation*}
        \max_{A \in R^{\ast}} (\mu(A) - \nu(A)) = \inf_{\pi \in \Pi(\mu,\nu)} (1 - \pi(R))
    \end{equation*}
    for any Borel probability measures $\mu,\nu$ on $S$, where
    \begin{equation*}
        R^{\ast} := \{A\in\mathcal B(S): \textnormal{ if }x\in A \textnormal{ and } (x,y)\in R, \textnormal{ then } y\in A\}.
    \end{equation*}
    The dual problem always admits optimizers, but the primal problem may not.
\end{theorem}

This result states that Thorisson's working hypothesis ``nearly holds'' for reflexive, transitive, $F_{\sigma}$ relations.
More precisely, it implies that the following distributional relation and its coupling counterpart are equivalent:
\begin{itemize}
    \item[(i)] We have $\mu(A) \le \nu(A)$ for all $A\in R^{\ast}$.
    \item[(ii)] For all $\varepsilon>0$, there exists a coupling $\pi_{\varepsilon}$ of $\mu$ and $\nu$ satisfying $\pi_{\varepsilon}(R) \ge 1-\varepsilon$.
\end{itemize}
However, the coupling counterpart statement (ii) cannot in general be improved.

The proof of Theorem~\ref{thm:main}, about which we now make a few remarks, is given in Section~\ref{sec:proof} and Section~\ref{sec:counter-ex}.
First, we note that the negative part is shown by directly constructing counterexamples in the setting of a simple $F_{\sigma}$ partial order on $[0,2]$, and then embedding this into further settings of interest.
Second, we note that the positive part involves a few steps, one of which is a general result in optimal transport that may be of independent interest.
To state it, recall that the usual optimal transport theory (e.g., \cite{Villani}) requires a cost function $c:S\times S\to\Rbb$ which is lower semicontinuous.
The following result shows that the Kantorovich dual problem is still well-behaved when $c$ lacks this regularity.

\begin{theorem}\label{thm:fn_reduction_to_one_var}
    Suppose $S$ is a Polish space and that $c:S\times S\to[0,\infty)$ satisfies 
    \begin{itemize}
        \item[(a)] $c(x,x) = 0$ for all $x\in S$,
        \item[(b)] $c(x,z) \le c(x,y) + c(y,z)$ for all $x,y,z\in S$, and
        \item[(c)] there exists $\{c_n\}_{n\in\mathbb{N}}$ such that $c_n:S\times S\to [0,\infty)$ is lower semicontinuous for each $n\in\mathbb{N}$, and $c_n\downarrow c$ pointwise as $n\to\infty$.
    \end{itemize}
    Then, for all Borel probability measures $\mu,\nu$ on $S$ we have
    \begin{equation}\label{eqn:one_var_function_equality}
        \sup_{\substack{\phi,\psi\in b\mathcal{B}(S)\\\phi\oplus \psi \le c}}\left(\int_{S} \phi\diff \mu + \int_{S}\psi\diff\nu\right) = \sup_{\substack{\phi\in b\mathcal{B}(S) \\ \phi\ominus\phi\le c}}\left(\int_{S} \phi\diff \mu - \int_{S}\phi\diff\nu\right)
    \end{equation}
    where $b\mathcal{B}(S)$ denotes the set of all bounded measurable functions from $S$ to $\Rbb$.
\end{theorem}

Note, in particular, that this result applies to cost functions $c$ which are upper semicontinuous, since every upper semicontinuous function can be written as a non-increasing limit of continuous functions.
However, we will apply the result to cost functions which are neither lower semicontinuous nor upper semicontinuous.

Lastly, we note that there are some remaining open questions arising from the connection we have identified between Thorisson's working hypothesis and optimal transport for zero-one-valued cost functions.
For instance, one may hope to understand cyclic monotonicity for such cost functions, hence \cite[Theorem~1]{BetterTransport} establishing optimality conditions for the resulting optimal transport problem;
this would imply some further structure of the various couplings statements appearing as statement (ii) above, as was previously done for equivalence relations in \cite[Proposition~3.12]{JaffeCoupling}.
For another example, one may hope to understand the existence of Monge solutions to the resulting optimal transport problem; this would imply that the couplings appearing in statement (ii) above can be achieved ``without additional randomness,'' as was previously done in \cite{Artstein} in the setting of measurable selections of random set-valued functions and in \cite{HummelJaffe} in the setting of germ couplings of Brownian motions with drift.


\section{Examples of Eventual Domination}\label{sec:examples}

In this section we give three examples of eventual domination arising in various disciplines of stochastic processes.
All of our examples can be understood in the following general framework: for some partial order $\le$ on a set $\Sigma$ we consider the pre-order $\preceq$ on $S:=\Sigma^{\mathbb N}$ whereby $x=\{x_n\}_{n\in\mathbb N}$ and $y=\{y_n\}_{n\in\mathbb N}$ satisfy $x\preceq y$ if and only if $x_n\le y_n$ for sufficiently large $n\in\mathbb N$.

For such relations, we claim that the monotone sets $R_\preceq^*$ constitute a strict subset of the tail $\sigma$-algebra $\mathcal T:=\bigcap_{n\in\mathbb N}\sigma(x_n,x_{n+1},\ldots)$, which is interesting from the perspective of Theorem~\ref{thm:main} because it says that the question of whether two stochastic processes can be coupled to satisfy eventual domination can be determined without using the full information of $\mathcal{T}$.
To see that $R_\preceq^* \subseteq \mathcal T$, notice that if $A \in R_\preceq^*$ and $x,y$ differ in only finitely many coordinates, then we have both $x \preceq y$ and $y \preceq x$, hence $x \in A$ if and only if $y \in A$, and this means $A\in\mathcal{T}$.
To see that the containment is strict, consider simply $\Sigma = \{0,1\}$ and $A:=\{x\in S : x_n = 0 \text{ for infinitely man } n\in\mathbb{N}\}$, which is in $\mathcal{T}$ but not $R_{\preceq}^{\ast}$.
In the rest of this section, we give various examples of events of interest in $R_{\preceq}^{\ast}$ in some concrete settings.

Our first and main example comes from the theory of branching random walks and related measure-valued process in discrete space and discrete time; such relations among branching random walks have been studied in \cite{Hutchcroft} and subsequent works \cite{BertacchiZucca,BertacchiCandelleroZucca} (also in earlier work \cite{JohnsonJunge}).

\begin{figure}
    \centering
    \begin{tikzpicture}[scale=0.85, arr/.style={thick, -stealth, shorten >=5pt,}]
		
		\begin{scope}
            \fill[black!10] (-2.5,-2.5) rectangle (2.5,0.25);
			\scaffold{1}
			\node[myred, font=\small] at (0,0.9) {$x$};
			\begin{scope}[myred!90, arr]
				\draw (0,0)--(-1,-1);   \draw (0,0)--(1,-1);
				\draw (-1,-1)--(-1,-2);  \draw (1,-1)--(0,-2);        
				\draw (-1,-2)--(0,-3);   \draw (0,-2)--(0,-3);        
				\draw (0,-3)--(-1,-4);   \draw (0,-3)--(1,-4);        
				\draw (-1,-4)--(-2,-5);  \draw (-1,-4)--(0,-5);       
			\end{scope}
			\occ{0}{0}{1}{myred}
			\occ{-1}{-1}{1}{myred}\occ{1}{-1}{1}{myred}
			\occ{-1}{-2}{1}{myred}\occ{0}{-2}{1}{myred}
			\occ{0}{-3}{2}{myred}                                    
			\occ{-1}{-4}{1}{myred}\occ{1}{-4}{1}{myred}
			\occ{-2}{-5}{1}{myred}\occ{0}{-5}{1}{myred}
			\cont{-2}{myred}\cont{0}{myred}
			\draw[{Stealth[length=4pt]}-{Stealth[length=4pt]}, black] (-2.35,-6.4) -- (2.45,-6.4);
			\node[black, font=\scriptsize] at (2.75,-6.4) {$E$};
		\end{scope}
		
		\node[scale=1] at (3.2,-0.5) {$\preceq$};
		
		\begin{scope}[shift={(6.4,0)}]
            \fill[black!10] (-2.5,-2.5) rectangle (2.5,0.25);
			\scaffold{1}
			\node[myblue, font=\small] at (0,0.9) {$y$};
			\begin{scope}[myblue!90, arr]
				\draw (0,0)--(-1,-1);   \draw (0,0)--(0,-1);
				\draw (-1,-1)--(-2,-2);  \draw (-1,-1)--(0,-2);  \draw (0,-1)--(1,-2); 
				\draw (-2,-2)--(-1,-3);  \draw (0,-2)--(0,-3);   \draw (1,-2)--(0,-3);  
				\draw (-1,-3)--(0,-4);   \draw (0,-3)--(-1,-4);  \draw (0,-3)--(1,-4);  
				\draw (-1,-4)--(-2,-5); \draw (1,-4)--(0,-5); \draw (1,-4)--(1,-5);
			\end{scope}
			\occ{0}{0}{1}{myblue}
			\occ{-1}{-1}{1}{myblue}\occ{0}{-1}{1}{myblue}
			\occ{-2}{-2}{1}{myblue}\occ{0}{-2}{1}{myblue}\occ{1}{-2}{1}{myblue}
			\occ{-1}{-3}{1}{myblue}\occ{0}{-3}{2}{myblue}           
			\occ{-1}{-4}{1}{myblue}\occ{0}{-4}{1}{myblue}\occ{1}{-4}{1}{myblue}
			\occ{-2}{-5}{1}{myblue}\occ{0}{-5}{1}{myblue}\occ{1}{-5}{1}{myblue}    
			\cont{-2}{myblue}\cont{0}{myblue}\cont{1}{myblue}
			\draw[{Stealth[length=4pt]}-{Stealth[length=4pt]}, black] (-2.35,-6.4) -- (2.45,-6.4);
			\node[black, font=\scriptsize] at (2.75,-6.4) {$E$};
		\end{scope}
		
	\end{tikzpicture}
    \caption{Eventual domination for sample paths of branching random walks (Example~\ref{ex:BRW}) in the space $E=\mathbb{Z}$.
    The measure-valued path $x$ is eventually dominated by $y$ since $x$ is eventwise less than $y$ after sufficiently large time.}
    \label{fig:ex-BRW}
\end{figure}

\begin{example}[branching random walks]\label{ex:BRW}
    To set things up, let $E$ be a countable state space, let $\Sigma:=\mathcal{M}(E)$ denote the space of non-negative finite measures on $E$, and set $S:=(\mathcal{M}(E))^{\mathbb N}$ be the space of discrete-time measure-valued paths on $E$.
    For each Markov transition kernel $P$ on $E$ and each probability measure $\lambda$ on $\mathbb N=\{0,1,\ldots\}$, we may consider the branching random walk corresponding to $(P,\lambda)$, which is the measure on $S$ representing the occupancy of a population undergoing Galton-Watson branching with offspring distribution $\lambda$ and where each individual experiences independent spatial motion according to $P$.
    See Figure~\ref{fig:ex-BRW} for an illustration in the case of $E=\mathbb{Z}$ and where $P$ is the transition kernel of a simple symmetric random walk.

    To define the relation of interest, we consider the partial order $\le$ of eventwise domination on $\Sigma=\mathcal{M}(E)$, and then we let $\preceq$ denote the corresponding pre-order of eventual domination on $S$.
    In terms of the branching populations, the eventual domination $x\preceq y$ means that, after sufficiently large time, every site in $E$ contains at least as many individuals from population $y$ as from population $x$.
    Thus, our Theorem~\ref{thm:main} states     
    \begin{equation*}
        \max_{A\in R_{\preceq}^{\ast}}(\mu(A) - \nu(A)) = \inf_{\pi\in\Pi(\mu,\nu)}(1-\pi(R_{\preceq}))
    \end{equation*}
    for all $\mu,\nu\in\mathcal{P}(S)$, although the infimum need not be achieved.
    (See \cite{BertacchiZucca} for related considerations that the primal problem need not be achieved.)

    The main reason for studying $\preceq$ and related orders in \cite{Hutchcroft} is to control the probability of extinction and non-extinction events in one branching random walk with respect to that of another.
    For example, note for all $B\subseteq E$ that events of the form $\{x\in S:x_n(B) \ge 1 \textnormal{ infinitely often}\}$ are in $R_{\preceq}^{\ast}$ and can be interpreted as the event that $x$ does not experience extinction in the region $B$; as such, the primal problem allows one to control all such non-extinction events uniformly; see \cite[Theorem~1.1]{Hutchcroft} for related results concerning the ``germ order'' on probability generating functions.
\end{example}

The remaining examples can be cast in the above form $S=\Sigma^{\mathbb N}$ by application of suitable bijections to the constructions below, but we find it more convenient to set them up by other means;
additionally, this highlights the generality of orderings that can be viewed as special cases of eventual domination.

Our second example concerns Ising models (equivalently, lattice gases) for which general background can be found in \cite{GeorgiiGibbs,PrestonFields}.
Orderings on configurations play a large role in this part of statistical mechanics, e.g., in the analysis of Glauber dynamics \cite[Chapter~15]{LevinPeres} and in the analysis of uniqueness of Gibbs measures \cite{PirogovSinai,Zahradnik,VanDenBergMaes}.

\begin{figure}
    \centering
    \begin{tikzpicture}[scale=0.5]
		
		\begin{scope}
            \fill[black!10] (-2.5,-2.5) rectangle (2.5,2.5);
			\gridbg
			\foreach \i/\j in {%
				-1/5,0/5,1/5, -1/4,0/4,1/4, -1/3,0/3,1/3,
				5/0,4/-1,4/0,4/1,3/-1,3/0,3/1,
				-2/-5,-1/-5,-3/-4,-2/-4,-1/-4,-5/-3,-4/-3,-3/-3,-2/-3,-3/-2,
				-1/2,0/2,1/2,2/2, -1/1,0/1,1/1,2/1,
				-2/0,-1/0,0/0,1/0,2/0, -2/-1,-1/-1,0/-1,1/-1,2/-1,
				-2/-2,-1/-2,0/-2,1/-2}{\spin{\i}{\j}{myred}}
			\boxmark
			\bcframe
			\foreach \i in {-1,0,1}{\bh{\i}{5.5}{myred}}
			\foreach \j in {0}{\bv{5.5}{\j}{myred}}
			\foreach \i in {-2,-1}{\bh{\i}{-5.5}{myred}}
			\foreach \j in {-3}{\bv{-5.5}{\j}{myred}}
			\node[myred, font=\small] at (0,6.4) {$x$};
		\end{scope}
        \node[black, font=\scriptsize] at (5,-5) {$G$};
		
		\node[scale=1] at (7,0) {$\preceq$};
		
		\begin{scope}[shift={(14,0)}]
			\fill[black!10] (-2.5,-2.5) rectangle (2.5,2.5);
            \gridbg
			\foreach \i/\j in {%
				-2/5,-1/5,0/5,1/5, -2/4,-1/4,0/4,1/4, -2/3,-1/3,0/3,1/3, -1/2,0/2,1/2,
				5/-1,5/0,5/1, 4/-1,4/0,4/1, 3/-1,3/0,3/1, 2/0,2/1,
				-3/-5,-2/-5,-1/-5, -4/-4,-3/-4,-2/-4,-1/-4,
				-5/-3,-4/-3,-3/-3,-2/-3, -4/-2,-3/-2,-2/-2}{\spin{\i}{\j}{myblue}}
			\boxmark
			\bcframe
			\foreach \i in {-2,-1,0,1}{\bh{\i}{5.5}{myblue}}
			\foreach \j in {-1,0,1}{\bv{5.5}{\j}{myblue}}
			\foreach \i in {-3,-2,-1}{\bh{\i}{-5.5}{myblue}}
			\foreach \j in {-3}{\bv{-5.5}{\j}{myblue}}
			\node[myblue, font=\small] at (0,6.4) {$y$};
            \node[black, font=\scriptsize] at (5,-5) {$G$};
		\end{scope}
		
	\end{tikzpicture}
    \caption{Eventual domination for Ising model configurations (Example~\ref{ex:Ising}), for the graph $G=\mathbb{Z}^2$.
    The configuration $x$ is eventually dominated by $y$ since $x$ is pointwise less than $y$ outside of a finite box.}
    \label{fig:ex-Ising}
\end{figure}

\begin{example}[Ising model]\label{ex:Ising}
    Fix an infinite, locally finite graph $G=(V,E)$, consider the set of spins $\Sigma:=\{0,1\}$, and define the configuration space $S:=\{0,1\}^{G}$.
    For convenience, let $\{\Lambda_n\}_{n\in\mathbb N}$ be a non-decreasing collection of finite subsets of $V$ with $\Lambda_n\uparrow V$ as $n\to\infty$, and consider an element $x=\{x_n\}_{n\in\mathbb N}\in S$ to be a sequence of functions $x_n:\Lambda_n\to \{0,1\}$ which are compatible in the sense that the restriction of $x_{n+1}$ to $\Lambda_{n}$ equals $x_n$.
    To define the relation, let $\le$ be the natural ordering on $\Sigma:=\{0,1\}$, and define $\preceq$ to be the relation of eventual domination; concretely, $x\preceq y$ means $y$ has a positive spin everywhere that $x$ contains a positive spin, except possibly inside a large finite region.
    See Figure~\ref{fig:ex-Ising} for an illustration in the case $G=\mathbb{Z}^2$ with $\Lambda_n=[-n,n]^2\cap\mathbb{Z}^2$.
    In this setting, Theorem~\ref{thm:main} yields     
    \begin{equation*}
        \max_{A\in R_{\preceq}^{\ast}}(\mu(A) - \nu(A)) = \inf_{\pi\in\Pi(\mu,\nu)}(1-\pi(R_{\preceq}))
    \end{equation*}
    for all $\mu,\nu\in\mathcal{P}(S)$, but as always the infimum need not be achieved.
    We note that the weak form of this duality (i.e., the statement that the primal upper bounds the dual) is essentially used in the \textit{disagreement percolation} argument of \cite{VanDenBergMaes} where a primal feasible coupling is explicitly constructed via site percolation on $G$.
\end{example}

The third example comes from stochastic calculus, for which we rely on material from standard sources like \cite{RevuzYor}.
Specifically, we consider an ordering of initial domination, which can be seen as a form of eventual domination under time inversion or reversal; this is a natural one-sided extension of the more commonly-studied equivalence relations of initial equality \cite{JaffeCoupling,HummelJaffe,MEXIT} and eventual equality 
\cite{Griffeath, Pitman, Goldstein, SverchkovSmirnov, HsuSturm, SchillingLevy1, BurdzyKendall, BenArousCranstonKendall, BanerjeeKendall2, Kuwada}.

\begin{figure}
    \centering
    \begin{tikzpicture}[scale=1.7]

  \begin{scope}
    \axes
    \draw[myred, line width=0.6pt, line join=round] plot coordinates {(0.000,-0.558) (0.004,-0.604) (0.008,-0.633) (0.012,-0.621) (0.016,-0.637) (0.020,-0.621) (0.024,-0.595) (0.028,-0.641) (0.032,-0.609) (0.036,-0.628) (0.040,-0.561) (0.044,-0.538) (0.048,-0.555) (0.052,-0.549) (0.056,-0.542) (0.060,-0.532) (0.064,-0.488) (0.068,-0.471) (0.072,-0.418) (0.076,-0.398) (0.080,-0.404) (0.084,-0.452) (0.088,-0.502) (0.092,-0.495) (0.096,-0.473) (0.100,-0.428) (0.104,-0.447) (0.108,-0.444) (0.112,-0.470) (0.116,-0.484) (0.120,-0.556) (0.124,-0.532) (0.128,-0.555) (0.132,-0.557) (0.136,-0.543) (0.140,-0.533) (0.144,-0.574) (0.148,-0.526) (0.152,-0.502) (0.156,-0.486) (0.160,-0.507) (0.164,-0.558) (0.168,-0.610) (0.172,-0.616) (0.176,-0.611) (0.180,-0.633) (0.184,-0.575) (0.188,-0.566) (0.192,-0.534) (0.196,-0.506) (0.200,-0.540) (0.204,-0.508) (0.208,-0.494) (0.212,-0.534) (0.216,-0.511) (0.220,-0.477) (0.224,-0.446) (0.228,-0.413) (0.232,-0.458) (0.236,-0.465) (0.240,-0.483) (0.244,-0.481) (0.248,-0.478) (0.252,-0.442) (0.256,-0.418) (0.260,-0.444) (0.264,-0.416) (0.268,-0.365) (0.272,-0.355) (0.276,-0.402) (0.280,-0.423) (0.284,-0.429) (0.288,-0.480) (0.292,-0.484) (0.296,-0.524) (0.300,-0.499) (0.304,-0.478) (0.308,-0.480) (0.312,-0.520) (0.316,-0.511) (0.320,-0.531) (0.324,-0.553) (0.328,-0.518) (0.332,-0.541) (0.336,-0.585) (0.340,-0.560) (0.344,-0.504) (0.348,-0.453) (0.352,-0.428) (0.356,-0.362) (0.360,-0.373) (0.364,-0.304) (0.368,-0.238) (0.372,-0.249) (0.376,-0.236) (0.380,-0.277) (0.384,-0.309) (0.388,-0.291) (0.392,-0.328) (0.396,-0.321) (0.400,-0.307) (0.404,-0.302) (0.408,-0.271) (0.412,-0.209) (0.416,-0.203) (0.420,-0.171) (0.424,-0.123) (0.428,-0.137) (0.432,-0.157) (0.436,-0.234) (0.440,-0.219) (0.444,-0.185) (0.448,-0.171) (0.452,-0.103) (0.456,-0.054) (0.460,-0.027) (0.464,-0.028) (0.468,-0.055) (0.472,-0.050) (0.476,-0.033) (0.480,-0.034) (0.484,-0.043) (0.488,-0.054) (0.492,-0.031) (0.496,-0.084) (0.500,-0.094) (0.504,-0.090) (0.508,-0.067) (0.512,-0.044) (0.516,-0.050) (0.520,-0.068) (0.524,-0.104) (0.528,-0.066) (0.532,-0.047) (0.536,-0.039) (0.540,-0.061) (0.544,-0.096) (0.548,-0.058) (0.552,-0.057) (0.556,-0.041) (0.560,0.017) (0.564,-0.031) (0.568,-0.054) (0.572,-0.042) (0.576,-0.021) (0.580,-0.053) (0.584,-0.054) (0.588,-0.072) (0.592,-0.090) (0.596,-0.077) (0.600,-0.069) (0.604,-0.031) (0.608,-0.011) (0.612,0.038) (0.616,0.013) (0.620,0.026) (0.624,-0.018) (0.628,-0.070) (0.632,-0.114) (0.636,-0.091) (0.640,-0.084) (0.644,-0.135) (0.648,-0.094) (0.652,-0.063) (0.656,-0.069) (0.660,-0.076) (0.664,-0.077) (0.668,-0.115) (0.672,-0.041) (0.676,-0.056) (0.680,-0.062) (0.684,-0.072) (0.688,-0.094) (0.692,-0.123) (0.696,-0.112) (0.700,-0.152) (0.704,-0.185) (0.708,-0.211) (0.712,-0.185) (0.716,-0.217) (0.720,-0.249) (0.724,-0.239) (0.728,-0.275) (0.732,-0.286) (0.736,-0.258) (0.740,-0.245) (0.744,-0.244) (0.748,-0.196) (0.752,-0.210) (0.756,-0.229) (0.760,-0.285) (0.764,-0.285) (0.768,-0.265) (0.772,-0.287) (0.776,-0.228) (0.780,-0.196) (0.784,-0.167) (0.788,-0.205) (0.792,-0.196) (0.796,-0.185) (0.800,-0.250) (0.804,-0.252) (0.808,-0.260) (0.812,-0.315) (0.816,-0.300) (0.820,-0.235) (0.824,-0.212) (0.828,-0.240) (0.832,-0.248) (0.836,-0.213) (0.840,-0.240) (0.844,-0.229) (0.848,-0.228) (0.852,-0.223) (0.856,-0.171) (0.860,-0.219) (0.864,-0.182) (0.868,-0.174) (0.872,-0.177) (0.876,-0.230) (0.880,-0.221) (0.884,-0.234) (0.888,-0.328) (0.892,-0.371) (0.896,-0.410) (0.900,-0.442) (0.904,-0.462) (0.908,-0.489) (0.912,-0.439) (0.916,-0.406) (0.920,-0.409) (0.924,-0.365) (0.928,-0.297) (0.932,-0.282) (0.936,-0.324) (0.940,-0.312) (0.944,-0.267) (0.948,-0.299) (0.952,-0.299) (0.956,-0.270) (0.960,-0.232) (0.964,-0.205) (0.968,-0.193) (0.972,-0.198) (0.976,-0.216) (0.980,-0.184) (0.984,-0.155) (0.988,-0.170) (0.992,-0.192) (0.996,-0.152) (1.000,-0.121) (1.004,-0.145) (1.008,-0.198) (1.012,-0.194) (1.016,-0.168) (1.020,-0.138) (1.024,-0.112) (1.028,-0.161) (1.032,-0.160) (1.036,-0.172) (1.040,-0.196) (1.044,-0.113) (1.048,-0.113) (1.052,-0.105) (1.056,-0.115) (1.060,-0.093) (1.064,-0.056) (1.068,-0.066) (1.072,-0.067) (1.076,-0.081) (1.080,-0.114) (1.084,-0.120) (1.088,-0.109) (1.092,-0.109) (1.096,-0.117) (1.100,-0.176) (1.104,-0.166) (1.108,-0.161) (1.112,-0.113) (1.116,-0.134) (1.120,-0.178) (1.124,-0.240) (1.128,-0.261) (1.132,-0.264) (1.136,-0.257) (1.140,-0.308) (1.144,-0.327) (1.148,-0.296) (1.152,-0.294) (1.156,-0.306) (1.160,-0.319) (1.164,-0.369) (1.168,-0.415) (1.172,-0.432) (1.176,-0.394) (1.180,-0.359) (1.184,-0.365) (1.188,-0.346) (1.192,-0.333) (1.196,-0.297) (1.200,-0.320) (1.204,-0.313) (1.208,-0.283) (1.212,-0.324) (1.216,-0.327) (1.220,-0.320) (1.224,-0.307) (1.228,-0.291) (1.232,-0.273) (1.236,-0.242) (1.240,-0.198) (1.244,-0.186) (1.248,-0.117) (1.252,-0.132) (1.256,-0.139) (1.260,-0.129) (1.264,-0.153) (1.268,-0.200) (1.272,-0.242) (1.276,-0.205) (1.280,-0.226) (1.284,-0.268) (1.288,-0.209) (1.292,-0.197) (1.296,-0.191) (1.300,-0.222) (1.304,-0.213) (1.308,-0.229) (1.312,-0.222) (1.316,-0.223) (1.320,-0.231) (1.324,-0.224) (1.328,-0.253) (1.332,-0.267) (1.336,-0.293) (1.340,-0.281) (1.344,-0.259) (1.348,-0.234) (1.352,-0.296) (1.356,-0.242) (1.360,-0.249) (1.364,-0.254) (1.368,-0.256) (1.372,-0.259) (1.376,-0.220) (1.380,-0.216) (1.384,-0.207) (1.388,-0.178) (1.392,-0.137) (1.396,-0.169) (1.400,-0.208) (1.404,-0.270) (1.408,-0.330) (1.412,-0.321) (1.416,-0.319) (1.420,-0.320) (1.424,-0.310) (1.428,-0.300) (1.432,-0.263) (1.436,-0.209) (1.440,-0.235) (1.444,-0.187) (1.448,-0.154) (1.452,-0.179) (1.456,-0.173) (1.460,-0.186) (1.464,-0.239) (1.468,-0.237) (1.472,-0.288) (1.476,-0.286) (1.480,-0.319) (1.484,-0.329) (1.488,-0.303) (1.492,-0.299) (1.496,-0.324) (1.500,-0.316) (1.504,-0.294) (1.508,-0.263) (1.512,-0.289) (1.516,-0.253) (1.520,-0.258) (1.524,-0.275) (1.528,-0.237) (1.532,-0.217) (1.536,-0.195) (1.540,-0.190) (1.544,-0.246) (1.548,-0.231) (1.552,-0.213) (1.556,-0.216) (1.560,-0.162) (1.564,-0.157) (1.568,-0.102) (1.572,-0.064) (1.576,-0.084) (1.580,-0.109) (1.584,-0.118) (1.588,-0.119) (1.592,-0.096) (1.596,-0.113) (1.600,-0.145) (1.604,-0.076) (1.608,-0.074) (1.612,-0.063) (1.616,-0.006) (1.620,-0.001) (1.624,-0.072) (1.628,-0.055) (1.632,-0.073) (1.636,-0.040) (1.640,-0.095) (1.644,-0.087) (1.648,-0.059) (1.652,-0.013) (1.656,0.020) (1.660,0.052) (1.664,0.054) (1.668,0.055) (1.672,0.026) (1.676,0.058) (1.680,0.110) (1.684,0.131) (1.688,0.106) (1.692,0.064) (1.696,0.050) (1.700,0.102) (1.704,0.109) (1.708,0.081) (1.712,0.072) (1.716,0.100) (1.720,0.082) (1.724,0.069) (1.728,0.055) (1.732,0.029) (1.736,0.092) (1.740,0.069) (1.744,0.080) (1.748,0.120) (1.752,0.119) (1.756,0.108) (1.760,0.106) (1.764,0.108) (1.768,0.104) (1.772,0.093) (1.776,0.083) (1.780,0.089) (1.784,0.147) (1.788,0.136) (1.792,0.164) (1.796,0.161) (1.800,0.133) (1.804,0.099) (1.808,0.112) (1.812,0.073) (1.816,0.090) (1.820,0.081) (1.824,0.137) (1.828,0.107) (1.832,0.144) (1.836,0.190) (1.840,0.219) (1.844,0.205) (1.848,0.218) (1.852,0.225) (1.856,0.223) (1.860,0.255) (1.864,0.255) (1.868,0.265) (1.872,0.302) (1.876,0.302) (1.880,0.293) (1.884,0.303) (1.888,0.290) (1.892,0.220) (1.896,0.207) (1.900,0.204) (1.904,0.170) (1.908,0.114) (1.912,0.131) (1.916,0.113) (1.920,0.135) (1.924,0.153) (1.928,0.203) (1.932,0.224) (1.936,0.255) (1.940,0.301) (1.944,0.293) (1.948,0.331) (1.952,0.371) (1.956,0.318) (1.960,0.287) (1.964,0.294) (1.968,0.351) (1.972,0.343) (1.976,0.305) (1.980,0.258) (1.984,0.270) (1.988,0.279) (1.992,0.261) (1.996,0.273) (2.000,0.263) (2.004,0.301) (2.008,0.259) (2.012,0.245) (2.016,0.282) (2.020,0.264) (2.024,0.239) (2.028,0.231) (2.032,0.290) (2.036,0.300) (2.040,0.346) (2.044,0.319) (2.048,0.314) (2.052,0.337) (2.056,0.346) (2.060,0.331) (2.064,0.262) (2.068,0.280) (2.072,0.304) (2.076,0.320) (2.080,0.268) (2.084,0.224) (2.088,0.206) (2.092,0.175) (2.096,0.208) (2.100,0.174) (2.104,0.164) (2.108,0.130) (2.112,0.075) (2.116,0.107) (2.120,0.112) (2.124,0.090) (2.128,0.091) (2.132,0.022) (2.136,0.072) (2.140,0.040) (2.144,-0.020) (2.148,-0.054) (2.152,-0.047) (2.156,-0.046) (2.160,-0.025) (2.164,-0.004) (2.168,-0.002) (2.172,0.030) (2.176,-0.027) (2.180,-0.013) (2.184,0.002) (2.188,0.013) (2.192,0.008) (2.196,-0.009) (2.200,-0.028) (2.204,-0.082) (2.208,-0.098) (2.212,-0.064) (2.216,-0.019) (2.220,-0.034) (2.224,-0.051) (2.228,-0.067) (2.232,-0.072) (2.236,-0.072) (2.240,-0.058) (2.244,-0.013) (2.248,-0.077) (2.252,-0.111) (2.256,-0.126) (2.260,-0.129) (2.264,-0.117) (2.268,-0.061) (2.272,-0.091) (2.276,-0.071) (2.280,-0.073) (2.284,-0.072) (2.288,-0.062) (2.292,-0.035) (2.296,-0.031) (2.300,-0.066) (2.304,-0.051) (2.308,-0.035) (2.312,-0.084) (2.316,-0.070) (2.320,-0.072) (2.324,-0.078) (2.328,-0.075) (2.332,-0.020) (2.336,-0.072) (2.340,-0.066) (2.344,-0.125) (2.348,-0.125) (2.352,-0.096) (2.356,-0.153) (2.360,-0.117) (2.364,-0.102) (2.368,-0.105) (2.372,-0.111) (2.376,-0.111) (2.380,-0.108) (2.384,-0.109) (2.388,-0.075) (2.392,-0.049) (2.396,-0.076) (2.400,-0.041) (2.404,0.018) (2.408,0.071) (2.412,0.082) (2.416,0.034) (2.420,0.020) (2.424,0.019) (2.428,0.018) (2.432,-0.025) (2.436,-0.004) (2.440,0.027) (2.444,-0.012) (2.448,-0.033) (2.452,-0.061) (2.456,-0.057) (2.460,-0.017) (2.464,-0.079) (2.468,-0.025) (2.472,0.009) (2.476,-0.028) (2.480,0.028) (2.484,0.069) (2.488,0.091) (2.492,0.106) (2.496,0.082) (2.500,0.102) (2.504,0.080) (2.508,0.087) (2.512,0.124) (2.516,0.117) (2.520,0.101) (2.524,0.122) (2.528,0.157) (2.532,0.193) (2.536,0.190) (2.540,0.239) (2.544,0.295) (2.548,0.289) (2.552,0.312) (2.556,0.260) (2.560,0.260) (2.564,0.231) (2.568,0.239) (2.572,0.224) (2.576,0.154) (2.580,0.177) (2.584,0.149) (2.588,0.190) (2.592,0.240) (2.596,0.262) (2.600,0.263) (2.604,0.275) (2.608,0.233) (2.612,0.240) (2.616,0.274) (2.620,0.300) (2.624,0.271) (2.628,0.322) (2.632,0.342) (2.636,0.316) (2.640,0.333) (2.644,0.307) (2.648,0.295) (2.652,0.300) (2.656,0.362) (2.660,0.444) (2.664,0.399) (2.668,0.400) (2.672,0.398) (2.676,0.422) (2.680,0.436) (2.684,0.419) (2.688,0.432) (2.692,0.414) (2.696,0.390) (2.700,0.410) (2.704,0.418) (2.708,0.401) (2.712,0.423) (2.716,0.412) (2.720,0.390) (2.724,0.388) (2.728,0.409) (2.732,0.391) (2.736,0.438) (2.740,0.432) (2.744,0.443) (2.748,0.469) (2.752,0.471) (2.756,0.452) (2.760,0.447) (2.764,0.478) (2.768,0.519) (2.772,0.541) (2.776,0.528) (2.780,0.499) (2.784,0.483) (2.788,0.470) (2.792,0.478) (2.796,0.479) (2.800,0.457) (2.804,0.480) (2.808,0.459) (2.812,0.463) (2.816,0.442) (2.820,0.476) (2.824,0.491) (2.828,0.499) (2.832,0.535) (2.836,0.567) (2.840,0.528) (2.844,0.537) (2.848,0.506) (2.852,0.490) (2.856,0.478) (2.860,0.435) (2.864,0.398) (2.868,0.402) (2.872,0.407) (2.876,0.425) (2.880,0.422) (2.884,0.410) (2.888,0.401) (2.892,0.392) (2.896,0.369) (2.900,0.373) (2.904,0.344) (2.908,0.312) (2.912,0.266) (2.916,0.269) (2.920,0.247) (2.924,0.239) (2.928,0.241) (2.932,0.213) (2.936,0.218) (2.940,0.219) (2.944,0.217) (2.948,0.157) (2.952,0.188) (2.956,0.204) (2.960,0.221) (2.964,0.274) (2.968,0.241) (2.972,0.301) (2.976,0.294) (2.980,0.322) (2.984,0.358) (2.988,0.327) (2.992,0.336) (2.996,0.318) (3.000,0.332) (3.004,0.287) (3.008,0.340) (3.012,0.354) (3.016,0.347) (3.020,0.321) (3.024,0.362) (3.028,0.347) (3.032,0.372) (3.036,0.336) (3.040,0.353) (3.044,0.295) (3.048,0.306) (3.052,0.275) (3.056,0.244) (3.060,0.272) (3.064,0.293) (3.068,0.274) (3.072,0.250) (3.076,0.283) (3.080,0.235) (3.084,0.225) (3.088,0.202) (3.092,0.201) (3.096,0.198) (3.100,0.191) (3.104,0.153) (3.108,0.159) (3.112,0.173) (3.116,0.158) (3.120,0.210) (3.124,0.195) (3.128,0.194) (3.132,0.253) (3.136,0.243) (3.140,0.239) (3.144,0.192) (3.148,0.191) (3.152,0.177) (3.156,0.185) (3.160,0.151) (3.164,0.187) (3.168,0.202) (3.172,0.228) (3.176,0.215) (3.180,0.160) (3.184,0.170) (3.188,0.168) (3.192,0.152) (3.196,0.201) (3.200,0.209) (3.204,0.212) (3.208,0.225) (3.212,0.235) (3.216,0.247) (3.220,0.271) (3.224,0.274) (3.228,0.260) (3.232,0.234) (3.236,0.241) (3.240,0.255) (3.244,0.279) (3.248,0.254) (3.252,0.207) (3.256,0.162) (3.260,0.207) (3.264,0.225) (3.268,0.224) (3.272,0.214) (3.276,0.281) (3.280,0.305) (3.284,0.327) (3.288,0.310) (3.292,0.345) (3.296,0.333) (3.300,0.309) (3.304,0.351) (3.308,0.361) (3.312,0.382) (3.316,0.336) (3.320,0.324) (3.324,0.297) (3.328,0.303) (3.332,0.299) (3.336,0.297) (3.340,0.273) (3.344,0.314) (3.348,0.273) (3.352,0.285) (3.356,0.282) (3.360,0.312) (3.364,0.280) (3.368,0.309) (3.372,0.281) (3.376,0.287) (3.380,0.278) (3.384,0.350) (3.388,0.307) (3.392,0.284) (3.396,0.224) (3.400,0.188) (3.404,0.179) (3.408,0.163) (3.412,0.133) (3.416,0.107) (3.420,0.096) (3.424,0.094) (3.428,0.141) (3.432,0.148) (3.436,0.142) (3.440,0.126) (3.444,0.142) (3.448,0.178) (3.452,0.165) (3.456,0.190) (3.460,0.201) (3.464,0.131) (3.468,0.187) (3.472,0.211) (3.476,0.164) (3.480,0.100) (3.484,0.102) (3.488,0.082) (3.492,0.094) (3.496,0.078) (3.500,0.113) (3.504,0.092) (3.508,0.102) (3.512,0.119) (3.516,0.108) (3.520,0.133) (3.524,0.134) (3.528,0.131) (3.532,0.123) (3.536,0.046) (3.540,0.013) (3.544,0.067) (3.548,0.054) (3.552,0.027) (3.556,0.019) (3.560,0.010) (3.564,0.073) (3.568,0.037) (3.572,0.013) (3.576,0.031) (3.580,-0.026) (3.584,-0.078) (3.588,-0.120) (3.592,-0.143) (3.596,-0.113) (3.600,-0.073) (3.604,-0.059) (3.608,0.001) (3.612,-0.060) (3.616,-0.028) (3.620,-0.051) (3.624,-0.037) (3.628,-0.047) (3.632,-0.093) (3.636,-0.058) (3.640,-0.015) (3.644,-0.008) (3.648,-0.034) (3.652,0.028) (3.656,-0.046) (3.660,-0.040) (3.664,-0.068) (3.668,-0.091) (3.672,-0.061) (3.676,-0.076) (3.680,-0.045) (3.684,-0.038) (3.688,0.006) (3.692,0.035) (3.696,0.062) (3.700,0.062) (3.704,0.060) (3.708,0.030) (3.712,0.044) (3.716,0.066) (3.720,0.045) (3.724,0.056) (3.728,0.032) (3.732,0.031) (3.736,-0.003) (3.740,-0.003) (3.744,0.030) (3.748,0.038) (3.752,-0.016) (3.756,0.004) (3.760,-0.007) (3.764,0.033) (3.768,0.054) (3.772,0.075) (3.776,0.065) (3.780,0.069) (3.784,0.046) (3.788,0.048) (3.792,0.072) (3.796,0.025) (3.800,0.014) (3.804,0.037) (3.808,0.004) (3.812,-0.001) (3.816,0.014) (3.820,0.024) (3.824,0.042) (3.828,0.066) (3.832,0.103) (3.836,0.093) (3.840,0.108) (3.844,0.120) (3.848,0.083) (3.852,0.070) (3.856,0.093) (3.860,0.104) (3.864,0.125) (3.868,0.124) (3.872,0.123) (3.876,0.152) (3.880,0.092) (3.884,0.040) (3.888,0.041) (3.892,0.049) (3.896,0.069) (3.900,0.090) (3.904,0.095) (3.908,0.059) (3.912,0.075) (3.916,0.081) (3.920,0.050) (3.924,-0.003) (3.928,-0.017) (3.932,0.003) (3.936,0.003) (3.940,0.038) (3.944,0.025) (3.948,0.003) (3.952,-0.016) (3.956,-0.016) (3.960,-0.008) (3.964,0.019) (3.968,0.015) (3.972,-0.022) (3.976,-0.048) (3.980,-0.041) (3.984,-0.046) (3.988,-0.041) (3.992,-0.060) (3.996,-0.074) (4.000,-0.090) (4.004,-0.110) (4.008,-0.105) (4.012,-0.055) (4.016,-0.056) (4.020,-0.052) (4.024,-0.111) (4.028,-0.096) (4.032,-0.104) (4.036,-0.065) (4.040,-0.049) (4.044,-0.006) (4.048,-0.036) (4.052,-0.015) (4.056,0.013) (4.060,0.001) (4.064,0.038) (4.068,-0.004) (4.072,0.004) (4.076,-0.009) (4.080,0.017) (4.084,-0.043) (4.088,-0.016) (4.092,0.036) (4.096,0.005) (4.100,0.034) (4.104,0.008) (4.108,0.036) (4.112,0.072) (4.116,0.038) (4.120,0.090) (4.124,0.088) (4.128,0.091) (4.132,0.117) (4.136,0.147) (4.140,0.151) (4.144,0.173) (4.148,0.195) (4.152,0.163) (4.156,0.174) (4.160,0.183) (4.164,0.187) (4.168,0.210) (4.172,0.197) (4.176,0.145) (4.180,0.167) (4.184,0.181) (4.188,0.184) (4.192,0.171) (4.196,0.241) (4.200,0.246) (4.204,0.217) (4.208,0.229) (4.212,0.247) (4.216,0.206) (4.220,0.229) (4.224,0.166) (4.228,0.156) (4.232,0.184) (4.236,0.220) (4.240,0.234) (4.244,0.260) (4.248,0.237) (4.252,0.242) (4.256,0.264) (4.260,0.213) (4.264,0.224) (4.268,0.245) (4.272,0.282) (4.276,0.308) (4.280,0.341) (4.284,0.369) (4.288,0.344) (4.292,0.318) (4.296,0.302) (4.300,0.266) (4.304,0.293) (4.308,0.347) (4.312,0.389) (4.316,0.398) (4.320,0.433) (4.324,0.437) (4.328,0.453) (4.332,0.490) (4.336,0.460) (4.340,0.444) (4.344,0.459) (4.348,0.551) (4.352,0.573) (4.356,0.599) (4.360,0.645) (4.364,0.633) (4.368,0.570) (4.372,0.536) (4.376,0.507) (4.380,0.562) (4.384,0.556) (4.388,0.549) (4.392,0.542) (4.396,0.534) (4.400,0.532) (4.404,0.533) (4.408,0.568) (4.412,0.573) (4.416,0.527) (4.420,0.551) (4.424,0.565) (4.428,0.552) (4.432,0.557) (4.436,0.536) (4.440,0.604) (4.444,0.637) (4.448,0.645) (4.452,0.618) (4.456,0.637) (4.460,0.655) (4.464,0.656) (4.468,0.647) (4.472,0.707) (4.476,0.695) (4.480,0.723) (4.484,0.760) (4.488,0.727) (4.492,0.734) (4.496,0.764) (4.500,0.766) (4.504,0.721) (4.508,0.740) (4.512,0.803) (4.516,0.818) (4.520,0.828) (4.524,0.806) (4.528,0.801) (4.532,0.815) (4.536,0.776) (4.540,0.748) (4.544,0.750) (4.548,0.750) (4.552,0.753) (4.556,0.751) (4.560,0.794) (4.564,0.791) (4.568,0.790) (4.572,0.765) (4.576,0.740) (4.580,0.737) (4.584,0.814) (4.588,0.849) (4.592,0.810) (4.596,0.858) (4.600,0.859) (4.604,0.821) (4.608,0.830) (4.612,0.801) (4.616,0.781) (4.620,0.752) (4.624,0.748) (4.628,0.687) (4.632,0.717) (4.636,0.768) (4.640,0.763) (4.644,0.766) (4.648,0.779) (4.652,0.773) (4.656,0.766) (4.660,0.788) (4.664,0.820) (4.668,0.809) (4.672,0.830) (4.676,0.841) (4.680,0.823) (4.684,0.822) (4.688,0.765) (4.692,0.732) (4.696,0.741) (4.700,0.759) (4.704,0.791) (4.708,0.836) (4.712,0.803) (4.716,0.813) (4.720,0.764) (4.724,0.762) (4.728,0.744) (4.732,0.758) (4.736,0.754) (4.740,0.756) (4.744,0.738) (4.748,0.738) (4.752,0.778) (4.756,0.802) (4.760,0.803) (4.764,0.753) (4.768,0.700) (4.772,0.698) (4.776,0.709) (4.780,0.727) (4.784,0.734) (4.788,0.711) (4.792,0.660) (4.796,0.684) (4.800,0.702) (4.804,0.738) (4.808,0.688) (4.812,0.707) (4.816,0.728) (4.820,0.696) (4.824,0.666) (4.828,0.635) (4.832,0.610) (4.836,0.606) (4.840,0.594) (4.844,0.634) (4.848,0.646) (4.852,0.628) (4.856,0.634) (4.860,0.612) (4.864,0.545) (4.868,0.554) (4.872,0.537) (4.876,0.575) (4.880,0.571) (4.884,0.613) (4.888,0.638) (4.892,0.592) (4.896,0.643) (4.900,0.630) (4.904,0.629) (4.908,0.595) (4.912,0.578) (4.916,0.658) (4.920,0.663) (4.924,0.555) (4.928,0.608) (4.932,0.615) (4.936,0.596) (4.940,0.616) (4.944,0.627) (4.948,0.607) (4.952,0.654) (4.956,0.672) (4.960,0.706) (4.964,0.713) (4.968,0.712) (4.972,0.717) (4.976,0.687) (4.980,0.661) (4.984,0.688) (4.988,0.686) (4.992,0.703) (4.996,0.703) (5.000,0.689) (5.004,0.721) (5.008,0.703) (5.012,0.674) (5.016,0.673) (5.020,0.705) (5.024,0.736) (5.028,0.702) (5.032,0.670) (5.036,0.652) (5.040,0.620) (5.044,0.618) (5.048,0.643) (5.052,0.632) (5.056,0.576) (5.060,0.519) (5.064,0.466) (5.068,0.449) (5.072,0.461) (5.076,0.410) (5.080,0.418) (5.084,0.435) (5.088,0.508) (5.092,0.471) (5.096,0.471) (5.100,0.491) (5.104,0.502) (5.108,0.526) (5.112,0.537) (5.116,0.558) (5.120,0.553) (5.124,0.549) (5.128,0.495) (5.132,0.556) (5.136,0.512) (5.140,0.516) (5.144,0.519) (5.148,0.492) (5.152,0.442) (5.156,0.449) (5.160,0.396) (5.164,0.429) (5.168,0.511) (5.172,0.554) (5.176,0.472) (5.180,0.485) (5.184,0.488) (5.188,0.481) (5.192,0.508) (5.196,0.493) (5.200,0.511) (5.204,0.490) (5.208,0.459) (5.212,0.477) (5.216,0.511) (5.220,0.493) (5.224,0.543) (5.228,0.538) (5.232,0.511) (5.236,0.488) (5.240,0.470) (5.244,0.492) (5.248,0.516) (5.252,0.520) (5.256,0.563) (5.260,0.536) (5.264,0.551) (5.268,0.561) (5.272,0.494) (5.276,0.506) (5.280,0.492) (5.284,0.541) (5.288,0.528) (5.292,0.571) (5.296,0.589) (5.300,0.544) (5.304,0.581) (5.308,0.604) (5.312,0.617) (5.316,0.614) (5.320,0.675) (5.324,0.669) (5.328,0.707) (5.332,0.730) (5.336,0.727) (5.340,0.708) (5.344,0.722) (5.348,0.723) (5.352,0.690) (5.356,0.706) (5.360,0.726) (5.364,0.732) (5.368,0.682) (5.372,0.703) (5.376,0.697) (5.380,0.727) (5.384,0.720) (5.388,0.703) (5.392,0.701) (5.396,0.701) (5.400,0.771) (5.404,0.742) (5.408,0.697) (5.412,0.694) (5.416,0.693) (5.420,0.638) (5.424,0.586) (5.428,0.559) (5.432,0.551) (5.436,0.541) (5.440,0.567) (5.444,0.566) (5.448,0.543) (5.452,0.540) (5.456,0.548) (5.460,0.567) (5.464,0.477) (5.468,0.478) (5.472,0.487) (5.476,0.532) (5.480,0.520) (5.484,0.523) (5.488,0.473) (5.492,0.497) (5.496,0.528) (5.500,0.500) (5.504,0.502) (5.508,0.487) (5.512,0.520) (5.516,0.512) (5.520,0.517) (5.524,0.524) (5.528,0.513) (5.532,0.519) (5.536,0.570) (5.540,0.570) (5.544,0.546) (5.548,0.583) (5.552,0.642) (5.556,0.672) (5.560,0.674) (5.564,0.674) (5.568,0.695) (5.572,0.728) (5.576,0.819) (5.580,0.894) (5.584,0.862) (5.588,0.888) (5.592,0.935) (5.596,0.957) (5.600,0.956) (5.604,0.937) (5.608,0.958) (5.612,0.981) (5.616,0.993) (5.620,0.966) (5.624,0.906) (5.628,0.960) (5.632,0.930) (5.636,0.926) (5.640,0.884) (5.644,0.914) (5.648,0.900) (5.652,0.903) (5.656,0.983) (5.660,1.017) (5.664,1.029) (5.668,0.973) (5.672,1.005) (5.676,1.037) (5.680,1.074) (5.684,1.085) (5.688,1.046) (5.692,1.029) (5.696,1.062) (5.700,1.091) (5.704,1.112) (5.708,1.186) (5.712,1.222) (5.716,1.231) (5.720,1.242) (5.724,1.189) (5.728,1.174) (5.732,1.206) (5.736,1.153) (5.740,1.165) (5.744,1.197) (5.748,1.189) (5.752,1.210) (5.756,1.181) (5.760,1.238) (5.764,1.228) (5.768,1.157) (5.772,1.173) (5.776,1.159) (5.780,1.212) (5.784,1.236) (5.788,1.225) (5.792,1.276) (5.796,1.266) (5.800,1.225) (5.804,1.262) (5.808,1.307) (5.812,1.299) (5.816,1.332) (5.820,1.390) (5.824,1.369) (5.828,1.271) (5.832,1.289) (5.836,1.281) (5.840,1.310) (5.844,1.317) (5.848,1.303) (5.852,1.270) (5.856,1.233) (5.860,1.214) (5.864,1.184) (5.868,1.208) (5.872,1.230) (5.876,1.210) (5.880,1.248) (5.884,1.236) (5.888,1.220) (5.892,1.247) (5.896,1.256) (5.900,1.237) (5.904,1.264) (5.908,1.317) (5.912,1.275) (5.916,1.266) (5.920,1.199) (5.924,1.195) (5.928,1.273) (5.932,1.281) (5.936,1.276) (5.940,1.281) (5.944,1.249) (5.948,1.255) (5.952,1.243) (5.956,1.218) (5.960,1.246) (5.964,1.238) (5.968,1.246) (5.972,1.225) (5.976,1.230) (5.980,1.248) (5.984,1.247) (5.988,1.296) (5.992,1.295) (5.996,1.308) (6.000,1.304)};
    \node[myred, font=\small] at (3.1,1.72) {$x$};

  \node[scale=0.8] at (3.45,1.72) {$\preceq$};

    \draw[myblue, line width=0.6pt, line join=round] plot coordinates {(0.000,0.212) (0.004,0.243) (0.008,0.224) (0.012,0.192) (0.016,0.194) (0.020,0.183) (0.024,0.163) (0.028,0.213) (0.032,0.204) (0.036,0.213) (0.040,0.173) (0.044,0.146) (0.048,0.138) (0.052,0.186) (0.056,0.174) (0.060,0.182) (0.064,0.119) (0.068,0.099) (0.072,0.092) (0.076,0.119) (0.080,0.135) (0.084,0.084) (0.088,0.115) (0.092,0.137) (0.096,0.162) (0.100,0.126) (0.104,0.180) (0.108,0.121) (0.112,0.118) (0.116,0.118) (0.120,0.135) (0.124,0.141) (0.128,0.122) (0.132,0.174) (0.136,0.111) (0.140,0.088) (0.144,0.074) (0.148,0.107) (0.152,0.118) (0.156,0.063) (0.160,0.132) (0.164,0.129) (0.168,0.168) (0.172,0.233) (0.176,0.221) (0.180,0.232) (0.184,0.183) (0.188,0.161) (0.192,0.211) (0.196,0.203) (0.200,0.235) (0.204,0.235) (0.208,0.258) (0.212,0.280) (0.216,0.347) (0.220,0.366) (0.224,0.372) (0.228,0.392) (0.232,0.373) (0.236,0.340) (0.240,0.260) (0.244,0.283) (0.248,0.291) (0.252,0.281) (0.256,0.311) (0.260,0.324) (0.264,0.315) (0.268,0.339) (0.272,0.358) (0.276,0.384) (0.280,0.422) (0.284,0.435) (0.288,0.459) (0.292,0.458) (0.296,0.428) (0.300,0.423) (0.304,0.425) (0.308,0.394) (0.312,0.393) (0.316,0.403) (0.320,0.375) (0.324,0.372) (0.328,0.379) (0.332,0.378) (0.336,0.419) (0.340,0.386) (0.344,0.455) (0.348,0.441) (0.352,0.497) (0.356,0.501) (0.360,0.563) (0.364,0.562) (0.368,0.554) (0.372,0.516) (0.376,0.499) (0.380,0.522) (0.384,0.528) (0.388,0.535) (0.392,0.535) (0.396,0.535) (0.400,0.557) (0.404,0.562) (0.408,0.537) (0.412,0.574) (0.416,0.574) (0.420,0.571) (0.424,0.577) (0.428,0.562) (0.432,0.543) (0.436,0.510) (0.440,0.542) (0.444,0.495) (0.448,0.502) (0.452,0.449) (0.456,0.405) (0.460,0.438) (0.464,0.437) (0.468,0.458) (0.472,0.460) (0.476,0.458) (0.480,0.485) (0.484,0.495) (0.488,0.531) (0.492,0.559) (0.496,0.529) (0.500,0.491) (0.504,0.536) (0.508,0.572) (0.512,0.553) (0.516,0.585) (0.520,0.547) (0.524,0.501) (0.528,0.537) (0.532,0.526) (0.536,0.486) (0.540,0.452) (0.544,0.490) (0.548,0.546) (0.552,0.535) (0.556,0.591) (0.560,0.557) (0.564,0.470) (0.568,0.462) (0.572,0.502) (0.576,0.507) (0.580,0.492) (0.584,0.517) (0.588,0.507) (0.592,0.517) (0.596,0.480) (0.600,0.505) (0.604,0.478) (0.608,0.460) (0.612,0.519) (0.616,0.521) (0.620,0.576) (0.624,0.531) (0.628,0.529) (0.632,0.542) (0.636,0.573) (0.640,0.617) (0.644,0.638) (0.648,0.674) (0.652,0.680) (0.656,0.716) (0.660,0.748) (0.664,0.765) (0.668,0.803) (0.672,0.808) (0.676,0.805) (0.680,0.780) (0.684,0.767) (0.688,0.784) (0.692,0.727) (0.696,0.743) (0.700,0.730) (0.704,0.759) (0.708,0.731) (0.712,0.734) (0.716,0.697) (0.720,0.680) (0.724,0.691) (0.728,0.748) (0.732,0.767) (0.736,0.773) (0.740,0.734) (0.744,0.751) (0.748,0.724) (0.752,0.756) (0.756,0.780) (0.760,0.811) (0.764,0.733) (0.768,0.693) (0.772,0.668) (0.776,0.659) (0.780,0.650) (0.784,0.643) (0.788,0.659) (0.792,0.688) (0.796,0.674) (0.800,0.638) (0.804,0.636) (0.808,0.642) (0.812,0.645) (0.816,0.607) (0.820,0.627) (0.824,0.639) (0.828,0.653) (0.832,0.618) (0.836,0.647) (0.840,0.648) (0.844,0.662) (0.848,0.674) (0.852,0.641) (0.856,0.635) (0.860,0.660) (0.864,0.664) (0.868,0.675) (0.872,0.642) (0.876,0.595) (0.880,0.589) (0.884,0.506) (0.888,0.535) (0.892,0.507) (0.896,0.560) (0.900,0.532) (0.904,0.524) (0.908,0.480) (0.912,0.430) (0.916,0.434) (0.920,0.448) (0.924,0.399) (0.928,0.413) (0.932,0.390) (0.936,0.357) (0.940,0.357) (0.944,0.389) (0.948,0.375) (0.952,0.352) (0.956,0.354) (0.960,0.327) (0.964,0.342) (0.968,0.319) (0.972,0.309) (0.976,0.242) (0.980,0.218) (0.984,0.160) (0.988,0.203) (0.992,0.174) (0.996,0.171) (1.000,0.171) (1.004,0.156) (1.008,0.138) (1.012,0.158) (1.016,0.191) (1.020,0.214) (1.024,0.219) (1.028,0.198) (1.032,0.196) (1.036,0.209) (1.040,0.261) (1.044,0.270) (1.048,0.244) (1.052,0.217) (1.056,0.208) (1.060,0.201) (1.064,0.246) (1.068,0.256) (1.072,0.278) (1.076,0.258) (1.080,0.214) (1.084,0.294) (1.088,0.364) (1.092,0.386) (1.096,0.415) (1.100,0.396) (1.104,0.442) (1.108,0.542) (1.112,0.540) (1.116,0.554) (1.120,0.583) (1.124,0.571) (1.128,0.566) (1.132,0.573) (1.136,0.548) (1.140,0.566) (1.144,0.529) (1.148,0.528) (1.152,0.525) (1.156,0.536) (1.160,0.518) (1.164,0.471) (1.168,0.461) (1.172,0.455) (1.176,0.496) (1.180,0.493) (1.184,0.519) (1.188,0.518) (1.192,0.518) (1.196,0.505) (1.200,0.454) (1.204,0.405) (1.208,0.342) (1.212,0.332) (1.216,0.380) (1.220,0.335) (1.224,0.364) (1.228,0.423) (1.232,0.424) (1.236,0.449) (1.240,0.441) (1.244,0.416) (1.248,0.390) (1.252,0.370) (1.256,0.388) (1.260,0.396) (1.264,0.374) (1.268,0.409) (1.272,0.416) (1.276,0.385) (1.280,0.395) (1.284,0.443) (1.288,0.487) (1.292,0.456) (1.296,0.475) (1.300,0.480) (1.304,0.522) (1.308,0.588) (1.312,0.506) (1.316,0.533) (1.320,0.492) (1.324,0.499) (1.328,0.547) (1.332,0.592) (1.336,0.572) (1.340,0.627) (1.344,0.614) (1.348,0.615) (1.352,0.575) (1.356,0.558) (1.360,0.521) (1.364,0.543) (1.368,0.520) (1.372,0.552) (1.376,0.530) (1.380,0.563) (1.384,0.588) (1.388,0.585) (1.392,0.605) (1.396,0.653) (1.400,0.706) (1.404,0.698) (1.408,0.697) (1.412,0.733) (1.416,0.828) (1.420,0.854) (1.424,0.819) (1.428,0.794) (1.432,0.787) (1.436,0.793) (1.440,0.749) (1.444,0.684) (1.448,0.691) (1.452,0.679) (1.456,0.680) (1.460,0.631) (1.464,0.610) (1.468,0.643) (1.472,0.634) (1.476,0.596) (1.480,0.619) (1.484,0.607) (1.488,0.636) (1.492,0.649) (1.496,0.642) (1.500,0.630) (1.504,0.662) (1.508,0.647) (1.512,0.662) (1.516,0.703) (1.520,0.685) (1.524,0.667) (1.528,0.665) (1.532,0.620) (1.536,0.633) (1.540,0.604) (1.544,0.587) (1.548,0.542) (1.552,0.579) (1.556,0.624) (1.560,0.629) (1.564,0.649) (1.568,0.662) (1.572,0.699) (1.576,0.765) (1.580,0.845) (1.584,0.865) (1.588,0.898) (1.592,0.879) (1.596,0.908) (1.600,0.949) (1.604,0.918) (1.608,0.911) (1.612,0.869) (1.616,0.866) (1.620,0.855) (1.624,0.812) (1.628,0.758) (1.632,0.749) (1.636,0.697) (1.640,0.722) (1.644,0.699) (1.648,0.672) (1.652,0.678) (1.656,0.687) (1.660,0.656) (1.664,0.678) (1.668,0.690) (1.672,0.707) (1.676,0.766) (1.680,0.758) (1.684,0.762) (1.688,0.713) (1.692,0.726) (1.696,0.697) (1.700,0.671) (1.704,0.669) (1.708,0.674) (1.712,0.667) (1.716,0.664) (1.720,0.650) (1.724,0.594) (1.728,0.646) (1.732,0.633) (1.736,0.640) (1.740,0.592) (1.744,0.569) (1.748,0.580) (1.752,0.568) (1.756,0.537) (1.760,0.526) (1.764,0.531) (1.768,0.514) (1.772,0.476) (1.776,0.514) (1.780,0.512) (1.784,0.572) (1.788,0.583) (1.792,0.536) (1.796,0.501) (1.800,0.471) (1.804,0.488) (1.808,0.438) (1.812,0.484) (1.816,0.465) (1.820,0.484) (1.824,0.513) (1.828,0.483) (1.832,0.463) (1.836,0.470) (1.840,0.507) (1.844,0.533) (1.848,0.526) (1.852,0.558) (1.856,0.559) (1.860,0.580) (1.864,0.548) (1.868,0.602) (1.872,0.598) (1.876,0.583) (1.880,0.582) (1.884,0.593) (1.888,0.619) (1.892,0.602) (1.896,0.546) (1.900,0.616) (1.904,0.596) (1.908,0.634) (1.912,0.612) (1.916,0.578) (1.920,0.608) (1.924,0.565) (1.928,0.477) (1.932,0.505) (1.936,0.525) (1.940,0.559) (1.944,0.570) (1.948,0.542) (1.952,0.513) (1.956,0.511) (1.960,0.518) (1.964,0.508) (1.968,0.518) (1.972,0.556) (1.976,0.566) (1.980,0.597) (1.984,0.583) (1.988,0.563) (1.992,0.535) (1.996,0.549) (2.000,0.528) (2.004,0.539) (2.008,0.507) (2.012,0.519) (2.016,0.557) (2.020,0.563) (2.024,0.562) (2.028,0.540) (2.032,0.528) (2.036,0.508) (2.040,0.540) (2.044,0.574) (2.048,0.598) (2.052,0.600) (2.056,0.622) (2.060,0.627) (2.064,0.575) (2.068,0.547) (2.072,0.508) (2.076,0.506) (2.080,0.486) (2.084,0.498) (2.088,0.453) (2.092,0.463) (2.096,0.484) (2.100,0.475) (2.104,0.447) (2.108,0.421) (2.112,0.429) (2.116,0.432) (2.120,0.358) (2.124,0.394) (2.128,0.397) (2.132,0.408) (2.136,0.386) (2.140,0.416) (2.144,0.469) (2.148,0.444) (2.152,0.448) (2.156,0.444) (2.160,0.435) (2.164,0.417) (2.168,0.463) (2.172,0.410) (2.176,0.417) (2.180,0.413) (2.184,0.367) (2.188,0.375) (2.192,0.333) (2.196,0.361) (2.200,0.387) (2.204,0.411) (2.208,0.461) (2.212,0.434) (2.216,0.428) (2.220,0.440) (2.224,0.502) (2.228,0.523) (2.232,0.533) (2.236,0.553) (2.240,0.579) (2.244,0.555) (2.248,0.593) (2.252,0.625) (2.256,0.616) (2.260,0.594) (2.264,0.542) (2.268,0.563) (2.272,0.557) (2.276,0.536) (2.280,0.503) (2.284,0.472) (2.288,0.497) (2.292,0.456) (2.296,0.482) (2.300,0.491) (2.304,0.502) (2.308,0.499) (2.312,0.460) (2.316,0.540) (2.320,0.577) (2.324,0.576) (2.328,0.582) (2.332,0.587) (2.336,0.508) (2.340,0.529) (2.344,0.516) (2.348,0.511) (2.352,0.450) (2.356,0.425) (2.360,0.391) (2.364,0.455) (2.368,0.443) (2.372,0.439) (2.376,0.467) (2.380,0.467) (2.384,0.487) (2.388,0.513) (2.392,0.517) (2.396,0.514) (2.400,0.474) (2.404,0.450) (2.408,0.394) (2.412,0.433) (2.416,0.439) (2.420,0.485) (2.424,0.492) (2.428,0.529) (2.432,0.490) (2.436,0.442) (2.440,0.446) (2.444,0.446) (2.448,0.474) (2.452,0.436) (2.456,0.391) (2.460,0.374) (2.464,0.385) (2.468,0.375) (2.472,0.353) (2.476,0.315) (2.480,0.301) (2.484,0.322) (2.488,0.320) (2.492,0.275) (2.496,0.269) (2.500,0.317) (2.504,0.317) (2.508,0.364) (2.512,0.344) (2.516,0.337) (2.520,0.406) (2.524,0.381) (2.528,0.375) (2.532,0.372) (2.536,0.392) (2.540,0.391) (2.544,0.454) (2.548,0.497) (2.552,0.533) (2.556,0.553) (2.560,0.584) (2.564,0.538) (2.568,0.557) (2.572,0.523) (2.576,0.520) (2.580,0.570) (2.584,0.601) (2.588,0.574) (2.592,0.587) (2.596,0.596) (2.600,0.630) (2.604,0.633) (2.608,0.603) (2.612,0.634) (2.616,0.578) (2.620,0.584) (2.624,0.572) (2.628,0.562) (2.632,0.614) (2.636,0.589) (2.640,0.572) (2.644,0.591) (2.648,0.523) (2.652,0.483) (2.656,0.514) (2.660,0.514) (2.664,0.515) (2.668,0.573) (2.672,0.524) (2.676,0.530) (2.680,0.520) (2.684,0.529) (2.688,0.499) (2.692,0.437) (2.696,0.434) (2.700,0.389) (2.704,0.371) (2.708,0.325) (2.712,0.341) (2.716,0.327) (2.720,0.361) (2.724,0.405) (2.728,0.379) (2.732,0.347) (2.736,0.344) (2.740,0.367) (2.744,0.398) (2.748,0.393) (2.752,0.402) (2.756,0.429) (2.760,0.440) (2.764,0.408) (2.768,0.421) (2.772,0.450) (2.776,0.414) (2.780,0.417) (2.784,0.430) (2.788,0.415) (2.792,0.405) (2.796,0.407) (2.800,0.411) (2.804,0.412) (2.808,0.392) (2.812,0.378) (2.816,0.411) (2.820,0.373) (2.824,0.369) (2.828,0.305) (2.832,0.322) (2.836,0.271) (2.840,0.290) (2.844,0.326) (2.848,0.324) (2.852,0.301) (2.856,0.309) (2.860,0.351) (2.864,0.388) (2.868,0.366) (2.872,0.393) (2.876,0.382) (2.880,0.404) (2.884,0.412) (2.888,0.422) (2.892,0.423) (2.896,0.414) (2.900,0.445) (2.904,0.404) (2.908,0.376) (2.912,0.374) (2.916,0.370) (2.920,0.385) (2.924,0.413) (2.928,0.396) (2.932,0.360) (2.936,0.427) (2.940,0.351) (2.944,0.353) (2.948,0.353) (2.952,0.335) (2.956,0.390) (2.960,0.394) (2.964,0.302) (2.968,0.313) (2.972,0.335) (2.976,0.386) (2.980,0.392) (2.984,0.353) (2.988,0.336) (2.992,0.315) (2.996,0.266) (3.000,0.274) (3.004,0.299) (3.008,0.327) (3.012,0.357) (3.016,0.371) (3.020,0.354) (3.024,0.418) (3.028,0.417) (3.032,0.379) (3.036,0.384) (3.040,0.347) (3.044,0.345) (3.048,0.367) (3.052,0.344) (3.056,0.335) (3.060,0.346) (3.064,0.340) (3.068,0.336) (3.072,0.303) (3.076,0.275) (3.080,0.281) (3.084,0.273) (3.088,0.275) (3.092,0.265) (3.096,0.208) (3.100,0.211) (3.104,0.164) (3.108,0.177) (3.112,0.185) (3.116,0.185) (3.120,0.181) (3.124,0.211) (3.128,0.231) (3.132,0.168) (3.136,0.152) (3.140,0.181) (3.144,0.148) (3.148,0.104) (3.152,0.127) (3.156,0.095) (3.160,0.124) (3.164,0.150) (3.168,0.128) (3.172,0.093) (3.176,0.156) (3.180,0.170) (3.184,0.170) (3.188,0.176) (3.192,0.223) (3.196,0.237) (3.200,0.222) (3.204,0.214) (3.208,0.243) (3.212,0.258) (3.216,0.261) (3.220,0.274) (3.224,0.215) (3.228,0.267) (3.232,0.264) (3.236,0.246) (3.240,0.168) (3.244,0.179) (3.248,0.108) (3.252,0.167) (3.256,0.080) (3.260,0.090) (3.264,0.058) (3.268,0.088) (3.272,-0.007) (3.276,-0.036) (3.280,-0.100) (3.284,-0.117) (3.288,-0.126) (3.292,-0.130) (3.296,-0.112) (3.300,-0.125) (3.304,-0.155) (3.308,-0.154) (3.312,-0.163) (3.316,-0.173) (3.320,-0.168) (3.324,-0.120) (3.328,-0.111) (3.332,-0.078) (3.336,-0.053) (3.340,-0.081) (3.344,-0.034) (3.348,-0.018) (3.352,-0.033) (3.356,-0.018) (3.360,0.019) (3.364,0.036) (3.368,0.072) (3.372,0.054) (3.376,-0.005) (3.380,-0.029) (3.384,-0.031) (3.388,-0.042) (3.392,-0.020) (3.396,-0.078) (3.400,-0.070) (3.404,-0.029) (3.408,-0.018) (3.412,-0.024) (3.416,0.011) (3.420,-0.043) (3.424,-0.028) (3.428,0.028) (3.432,0.013) (3.436,0.025) (3.440,0.024) (3.444,0.021) (3.448,0.041) (3.452,0.072) (3.456,0.018) (3.460,-0.067) (3.464,-0.057) (3.468,-0.083) (3.472,-0.108) (3.476,-0.116) (3.480,-0.119) (3.484,-0.126) (3.488,-0.100) (3.492,-0.118) (3.496,-0.099) (3.500,-0.121) (3.504,-0.133) (3.508,-0.148) (3.512,-0.103) (3.516,-0.077) (3.520,-0.061) (3.524,-0.076) (3.528,-0.095) (3.532,-0.078) (3.536,-0.050) (3.540,-0.025) (3.544,-0.036) (3.548,-0.040) (3.552,-0.065) (3.556,-0.044) (3.560,-0.027) (3.564,0.010) (3.568,0.030) (3.572,0.053) (3.576,0.035) (3.580,-0.005) (3.584,0.028) (3.588,0.028) (3.592,0.056) (3.596,-0.002) (3.600,0.002) (3.604,-0.015) (3.608,0.001) (3.612,-0.085) (3.616,-0.092) (3.620,-0.108) (3.624,-0.136) (3.628,-0.177) (3.632,-0.132) (3.636,-0.128) (3.640,-0.133) (3.644,-0.136) (3.648,-0.119) (3.652,-0.102) (3.656,-0.065) (3.660,-0.133) (3.664,-0.210) (3.668,-0.225) (3.672,-0.276) (3.676,-0.254) (3.680,-0.271) (3.684,-0.199) (3.688,-0.203) (3.692,-0.242) (3.696,-0.255) (3.700,-0.219) (3.704,-0.204) (3.708,-0.205) (3.712,-0.229) (3.716,-0.275) (3.720,-0.295) (3.724,-0.233) (3.728,-0.232) (3.732,-0.273) (3.736,-0.325) (3.740,-0.318) (3.744,-0.279) (3.748,-0.253) (3.752,-0.277) (3.756,-0.328) (3.760,-0.305) (3.764,-0.332) (3.768,-0.300) (3.772,-0.352) (3.776,-0.356) (3.780,-0.387) (3.784,-0.410) (3.788,-0.366) (3.792,-0.417) (3.796,-0.472) (3.800,-0.466) (3.804,-0.445) (3.808,-0.456) (3.812,-0.417) (3.816,-0.419) (3.820,-0.491) (3.824,-0.501) (3.828,-0.468) (3.832,-0.486) (3.836,-0.507) (3.840,-0.472) (3.844,-0.398) (3.848,-0.387) (3.852,-0.394) (3.856,-0.348) (3.860,-0.332) (3.864,-0.308) (3.868,-0.312) (3.872,-0.285) (3.876,-0.290) (3.880,-0.234) (3.884,-0.219) (3.888,-0.225) (3.892,-0.177) (3.896,-0.191) (3.900,-0.279) (3.904,-0.320) (3.908,-0.343) (3.912,-0.326) (3.916,-0.290) (3.920,-0.312) (3.924,-0.286) (3.928,-0.280) (3.932,-0.298) (3.936,-0.290) (3.940,-0.350) (3.944,-0.373) (3.948,-0.335) (3.952,-0.367) (3.956,-0.344) (3.960,-0.358) (3.964,-0.355) (3.968,-0.408) (3.972,-0.386) (3.976,-0.398) (3.980,-0.430) (3.984,-0.313) (3.988,-0.384) (3.992,-0.385) (3.996,-0.378) (4.000,-0.332) (4.004,-0.375) (4.008,-0.357) (4.012,-0.342) (4.016,-0.340) (4.020,-0.326) (4.024,-0.336) (4.028,-0.355) (4.032,-0.322) (4.036,-0.357) (4.040,-0.340) (4.044,-0.366) (4.048,-0.393) (4.052,-0.348) (4.056,-0.355) (4.060,-0.395) (4.064,-0.400) (4.068,-0.386) (4.072,-0.414) (4.076,-0.412) (4.080,-0.448) (4.084,-0.505) (4.088,-0.474) (4.092,-0.457) (4.096,-0.416) (4.100,-0.345) (4.104,-0.286) (4.108,-0.274) (4.112,-0.274) (4.116,-0.291) (4.120,-0.244) (4.124,-0.233) (4.128,-0.237) (4.132,-0.270) (4.136,-0.287) (4.140,-0.294) (4.144,-0.303) (4.148,-0.308) (4.152,-0.295) (4.156,-0.241) (4.160,-0.217) (4.164,-0.207) (4.168,-0.192) (4.172,-0.158) (4.176,-0.106) (4.180,-0.144) (4.184,-0.168) (4.188,-0.151) (4.192,-0.090) (4.196,-0.066) (4.200,-0.037) (4.204,-0.055) (4.208,-0.032) (4.212,0.007) (4.216,-0.016) (4.220,-0.013) (4.224,0.037) (4.228,0.050) (4.232,0.035) (4.236,0.052) (4.240,0.047) (4.244,0.060) (4.248,0.099) (4.252,0.091) (4.256,0.110) (4.260,0.106) (4.264,0.045) (4.268,-0.002) (4.272,-0.001) (4.276,0.028) (4.280,0.068) (4.284,0.003) (4.288,0.016) (4.292,-0.026) (4.296,-0.111) (4.300,-0.125) (4.304,-0.156) (4.308,-0.130) (4.312,-0.123) (4.316,-0.163) (4.320,-0.185) (4.324,-0.151) (4.328,-0.137) (4.332,-0.114) (4.336,-0.141) (4.340,-0.173) (4.344,-0.135) (4.348,-0.097) (4.352,-0.110) (4.356,-0.139) (4.360,-0.119) (4.364,-0.103) (4.368,-0.164) (4.372,-0.146) (4.376,-0.121) (4.380,-0.143) (4.384,-0.096) (4.388,-0.089) (4.392,-0.078) (4.396,-0.118) (4.400,-0.143) (4.404,-0.148) (4.408,-0.135) (4.412,-0.123) (4.416,-0.135) (4.420,-0.118) (4.424,-0.163) (4.428,-0.169) (4.432,-0.191) (4.436,-0.213) (4.440,-0.206) (4.444,-0.205) (4.448,-0.187) (4.452,-0.154) (4.456,-0.212) (4.460,-0.185) (4.464,-0.201) (4.468,-0.198) (4.472,-0.201) (4.476,-0.187) (4.480,-0.218) (4.484,-0.250) (4.488,-0.277) (4.492,-0.293) (4.496,-0.288) (4.500,-0.283) (4.504,-0.297) (4.508,-0.270) (4.512,-0.279) (4.516,-0.307) (4.520,-0.291) (4.524,-0.339) (4.528,-0.324) (4.532,-0.297) (4.536,-0.296) (4.540,-0.237) (4.544,-0.255) (4.548,-0.312) (4.552,-0.267) (4.556,-0.268) (4.560,-0.249) (4.564,-0.245) (4.568,-0.265) (4.572,-0.245) (4.576,-0.289) (4.580,-0.261) (4.584,-0.280) (4.588,-0.265) (4.592,-0.242) (4.596,-0.226) (4.600,-0.219) (4.604,-0.176) (4.608,-0.173) (4.612,-0.182) (4.616,-0.165) (4.620,-0.190) (4.624,-0.224) (4.628,-0.282) (4.632,-0.268) (4.636,-0.281) (4.640,-0.237) (4.644,-0.249) (4.648,-0.222) (4.652,-0.225) (4.656,-0.235) (4.660,-0.222) (4.664,-0.232) (4.668,-0.224) (4.672,-0.249) (4.676,-0.288) (4.680,-0.226) (4.684,-0.219) (4.688,-0.246) (4.692,-0.245) (4.696,-0.256) (4.700,-0.250) (4.704,-0.282) (4.708,-0.276) (4.712,-0.271) (4.716,-0.255) (4.720,-0.260) (4.724,-0.264) (4.728,-0.236) (4.732,-0.177) (4.736,-0.167) (4.740,-0.142) (4.744,-0.121) (4.748,-0.134) (4.752,-0.062) (4.756,-0.097) (4.760,-0.167) (4.764,-0.164) (4.768,-0.181) (4.772,-0.203) (4.776,-0.247) (4.780,-0.241) (4.784,-0.264) (4.788,-0.280) (4.792,-0.320) (4.796,-0.284) (4.800,-0.310) (4.804,-0.319) (4.808,-0.332) (4.812,-0.349) (4.816,-0.351) (4.820,-0.379) (4.824,-0.337) (4.828,-0.339) (4.832,-0.377) (4.836,-0.359) (4.840,-0.324) (4.844,-0.361) (4.848,-0.412) (4.852,-0.424) (4.856,-0.435) (4.860,-0.427) (4.864,-0.465) (4.868,-0.435) (4.872,-0.422) (4.876,-0.477) (4.880,-0.475) (4.884,-0.500) (4.888,-0.521) (4.892,-0.554) (4.896,-0.504) (4.900,-0.486) (4.904,-0.504) (4.908,-0.525) (4.912,-0.586) (4.916,-0.499) (4.920,-0.535) (4.924,-0.493) (4.928,-0.487) (4.932,-0.510) (4.936,-0.491) (4.940,-0.508) (4.944,-0.516) (4.948,-0.489) (4.952,-0.523) (4.956,-0.474) (4.960,-0.542) (4.964,-0.510) (4.968,-0.520) (4.972,-0.548) (4.976,-0.601) (4.980,-0.629) (4.984,-0.625) (4.988,-0.659) (4.992,-0.660) (4.996,-0.631) (5.000,-0.654) (5.004,-0.651) (5.008,-0.688) (5.012,-0.693) (5.016,-0.666) (5.020,-0.723) (5.024,-0.759) (5.028,-0.740) (5.032,-0.699) (5.036,-0.682) (5.040,-0.660) (5.044,-0.671) (5.048,-0.698) (5.052,-0.679) (5.056,-0.713) (5.060,-0.708) (5.064,-0.735) (5.068,-0.721) (5.072,-0.732) (5.076,-0.792) (5.080,-0.783) (5.084,-0.782) (5.088,-0.802) (5.092,-0.821) (5.096,-0.842) (5.100,-0.910) (5.104,-0.894) (5.108,-0.891) (5.112,-0.835) (5.116,-0.781) (5.120,-0.772) (5.124,-0.851) (5.128,-0.878) (5.132,-0.864) (5.136,-0.899) (5.140,-0.867) (5.144,-0.855) (5.148,-0.858) (5.152,-0.887) (5.156,-0.866) (5.160,-0.857) (5.164,-0.909) (5.168,-0.890) (5.172,-0.937) (5.176,-0.941) (5.180,-0.874) (5.184,-0.911) (5.188,-0.873) (5.192,-0.938) (5.196,-0.992) (5.200,-1.042) (5.204,-1.075) (5.208,-1.119) (5.212,-1.107) (5.216,-1.072) (5.220,-1.093) (5.224,-1.134) (5.228,-1.149) (5.232,-1.165) (5.236,-1.218) (5.240,-1.234) (5.244,-1.252) (5.248,-1.247) (5.252,-1.199) (5.256,-1.208) (5.260,-1.208) (5.264,-1.183) (5.268,-1.123) (5.272,-1.129) (5.276,-1.156) (5.280,-1.152) (5.284,-1.125) (5.288,-1.118) (5.292,-1.088) (5.296,-1.102) (5.300,-1.111) (5.304,-1.138) (5.308,-1.139) (5.312,-1.155) (5.316,-1.167) (5.320,-1.215) (5.324,-1.223) (5.328,-1.236) (5.332,-1.254) (5.336,-1.265) (5.340,-1.286) (5.344,-1.280) (5.348,-1.275) (5.352,-1.308) (5.356,-1.294) (5.360,-1.312) (5.364,-1.311) (5.368,-1.294) (5.372,-1.335) (5.376,-1.390) (5.380,-1.373) (5.384,-1.373) (5.388,-1.312) (5.392,-1.304) (5.396,-1.312) (5.400,-1.294) (5.404,-1.288) (5.408,-1.293) (5.412,-1.233) (5.416,-1.215) (5.420,-1.226) (5.424,-1.251) (5.428,-1.235) (5.432,-1.223) (5.436,-1.222) (5.440,-1.212) (5.444,-1.192) (5.448,-1.150) (5.452,-1.079) (5.456,-1.128) (5.460,-1.154) (5.464,-1.173) (5.468,-1.157) (5.472,-1.131) (5.476,-1.066) (5.480,-1.078) (5.484,-1.085) (5.488,-1.082) (5.492,-1.078) (5.496,-1.110) (5.500,-1.099) (5.504,-1.083) (5.508,-1.103) (5.512,-1.117) (5.516,-1.064) (5.520,-1.119) (5.524,-1.109) (5.528,-1.075) (5.532,-1.051) (5.536,-1.031) (5.540,-1.013) (5.544,-1.020) (5.548,-1.007) (5.552,-0.983) (5.556,-0.940) (5.560,-0.971) (5.564,-1.045) (5.568,-1.065) (5.572,-1.095) (5.576,-1.091) (5.580,-1.100) (5.584,-1.103) (5.588,-1.062) (5.592,-1.093) (5.596,-1.087) (5.600,-1.089) (5.604,-1.109) (5.608,-1.133) (5.612,-1.158) (5.616,-1.145) (5.620,-1.145) (5.624,-1.120) (5.628,-1.051) (5.632,-1.031) (5.636,-1.038) (5.640,-1.029) (5.644,-1.053) (5.648,-1.083) (5.652,-1.095) (5.656,-1.082) (5.660,-1.114) (5.664,-1.122) (5.668,-1.191) (5.672,-1.228) (5.676,-1.185) (5.680,-1.165) (5.684,-1.091) (5.688,-1.089) (5.692,-1.068) (5.696,-1.042) (5.700,-1.069) (5.704,-1.073) (5.708,-1.043) (5.712,-1.069) (5.716,-1.062) (5.720,-0.995) (5.724,-1.003) (5.728,-1.019) (5.732,-1.011) (5.736,-1.014) (5.740,-0.945) (5.744,-0.989) (5.748,-1.016) (5.752,-0.977) (5.756,-0.986) (5.760,-0.990) (5.764,-0.924) (5.768,-0.966) (5.772,-0.952) (5.776,-0.940) (5.780,-0.961) (5.784,-0.982) (5.788,-0.986) (5.792,-0.958) (5.796,-0.977) (5.800,-1.031) (5.804,-1.040) (5.808,-1.066) (5.812,-1.029) (5.816,-1.002) (5.820,-0.965) (5.824,-0.990) (5.828,-1.022) (5.832,-0.980) (5.836,-0.979) (5.840,-0.973) (5.844,-0.947) (5.848,-0.963) (5.852,-0.930) (5.856,-0.969) (5.860,-1.036) (5.864,-1.073) (5.868,-1.047) (5.872,-1.035) (5.876,-1.125) (5.880,-1.097) (5.884,-1.127) (5.888,-1.161) (5.892,-1.124) (5.896,-1.124) (5.900,-1.094) (5.904,-1.173) (5.908,-1.180) (5.912,-1.230) (5.916,-1.215) (5.920,-1.221) (5.924,-1.207) (5.928,-1.144) (5.932,-1.123) (5.936,-1.117) (5.940,-1.083) (5.944,-1.064) (5.948,-1.040) (5.952,-1.005) (5.956,-0.948) (5.960,-0.974) (5.964,-0.949) (5.968,-0.948) (5.972,-0.955) (5.976,-0.977) (5.980,-0.996) (5.984,-1.029) (5.988,-1.001) (5.992,-1.015) (5.996,-1.014) (6.000,-1.013)};
    \node[myblue, font=\small] at (3.8,1.72) {$y$};
  \end{scope}

    \end{tikzpicture}
    \caption{Initial domination for sample paths of diffusions (Example~\ref{ex:diffusions}).
    The path $x$ is initially dominated by $y$ since $x$ is pointwise less than $y$ up to some strictly positive time.}
    \label{fig:ex-diffusions}
\end{figure}

\begin{example}[diffusions]\label{ex:diffusions}
    Let $S:=C([0,\infty);\Rbb)$ denote the space of continuous functions endowed with the topology of uniform convergence on compact sets.
    Let $\preceq$ denote the relation on $S$ where $x\preceq y$ means $x_t\le y_t$ for all sufficiently small $t>0$; in other words, the sample path $y$ stays entirely above the sample path $x$ for some small amount of time;
    see Figure~\ref{fig:ex-diffusions} for an illustration.
    Our Theorem~\ref{thm:main} then states
    
    \begin{equation*}
        \max_{A\in R_{\preceq}^{\ast}}(\mu(A) - \nu(A)) = \inf_{\pi\in\Pi(\mu,\nu)}(1-\pi(R_{\preceq}))
    \end{equation*}
    for all $\mu,\nu\in\mathcal{P}(S)$, where the infimum can fail to be achieved.
    However, we emphasize that, when $\mu,\nu$ are the laws of diffusions admitting strong solutions, then one can often construct explicitly primal feasible couplings (albeit, suboptimal ones) by coupling their driving Brownian motions; this allows one to uniformly control the dual side, which contains many events of interest.
\end{example}

\section{Proof of Positive Part of Theorem~\ref{thm:main}}\label{sec:proof}

This section contains the proof our main result, as well as some intermediate results of independent interest.
Throughout, we let $R\subseteq S\times S$ denote a reflexive, transitive, $F_{\sigma}$ relation, and we fix arbitrary Borel probability measures $\mu,\nu$ on $S$.
The remaining subsections divide the proof into a few  different steps: eliminating one of the two dual variables (Subsection~\ref{subsec:dual-elim}), proving the existence of a dual maximizer (Subsection~\ref{subsec:maximizer}), and converting this dual maximizer into a maximizing indicator function (Subsection~\ref{subsec:indicator}).

\subsection{Dual variable elimination}\label{subsec:dual-elim}

Note that $R$ being $F_{\sigma}$ implies that there exists a non-decreasing sequence $\{R_n\}_{n\in\mathbb{N}}$ of closed sets in $S\times S$ such that $R = \bigcup_{n\in\mathbb{N}}R_n$.
Then, let us write $c=1-\ind_R$ and $c_n = 1-\ind_{R_n}$ for each $n\in\mathbb{N}$.
It follows that the hypotheses of Theorem~\ref{thm:fn_reduction_to_one_var} are satisfied, so, by combining with general results on Kantorovich duality (e.g., \cite[Corollary~2.3.9]{RachevRdorf}), we have
\begin{equation}\label{eqn:one-dual-var}
    \sup_{\substack{\phi\in b\mathcal{B}(S)\\\phi\ominus \phi\le c}}\left(\int_{S}\phi\diff \mu - \int_{S}\phi \diff \nu\right)
    = \inf_{\pi\in\Pi(\mu,\nu)}(1-\pi(R)).
\end{equation}
In the remainder of this subsection we prove Theorem~\ref{thm:fn_reduction_to_one_var}, hence establishing~\eqref{eqn:one-dual-var}.
Note that this shows that the Kantorovich dual can in fact be taken to have one dual variable rather than two.

\begin{proof}[Proof of Theorem~\ref{thm:fn_reduction_to_one_var}]
Let $\phi,\psi:S\to\Rbb$ be arbitrary bounded, Borel measurable functions satisfying $\phi(x)+\psi(y)\le c(x,y)$ for all $x,y\in S$, and write $\psi^c(x) \coloneqq \inf_{y \in S} (c(x,y) - \psi(y))$ for the \textit{$c$-transform of $\psi$}.
    Now note that property (a) implies $\psi^c \leq -\psi$ pointwise, and that property (b) implies $\psi^{cc} = -\psi^c$.
    Moreover, by standard optimal transport considerations, we have
    \begin{equation}\label{eqn:c_transform_dual_inequality}
        \int \phi \, d\mu + \int \psi \, d\nu \le \int \psi^c \, d\mu + \int \psi \, d\nu \le \int \psi^c \, d\mu - \int \psi^c \, d\nu.
    \end{equation}
    While it may appear that Theorem~\ref{thm:fn_reduction_to_one_var} follows immediately from \eqref{eqn:c_transform_dual_inequality}, there is a problem in that $\psi^c$ need not be Borel measurable; indeed, if $\psi = \mathbf \ind_B$ and $c = 1-\ind_R$, then $\psi^c = -\mathbf \ind_{p((B\times S)\cap R)}$ where $p:S\times S\to S$ denotes the first-coordinate projection, in which case $p((B\times S)\cap R)$ analytic but possibly non-measurable.

    We claim that $\psi^c$ is, in general, universally measurable, so that $\int \psi^c \diff\mu, \int \psi^c \diff\nu$ are well-defined via the $\mu$- and $\nu$-completions of $\mathcal{B}(S)$, respectively.
    To show this, note that the strict sub-level sets can be written as
        \begin{align*}
            (\psi^c)^{-1}((-\infty,\alpha)) &= \left\{ x \in S : \inf_{y \in S} (c(x,y) - \psi(y)) < \alpha \right\} \\
            &= \{ x \in S : \textnormal{there exists } y \in S \text{ such that } c(x,y) - \psi(y) < \alpha \} \\
            &= p(\{ (x,y) \in S \times S : c(x,y) - \psi(y) < \alpha \}).
        \end{align*}
        As projections of Borel sets, these sub-level sets are universally measurable. (See \cite[Theorem~A1.2]{Kallenberg}.)

        Now let us show how to get around this measurability problem by directly constructing a Borel-measurable function $f$ with $f(x) - f(y) \leq c(x,y)$ pointwise, whose objective value on the right side of \eqref{eqn:one_var_function_equality} is within $\varepsilon>0$ of the objective value of $(\phi,\psi)$ on the left side of \eqref{eqn:one_var_function_equality}, where $\varepsilon>0$ is fixed and arbitrary. First, without loss of generality, we may assume that $\psi \leq 0$, so that $\psi^c \geq 0$. Now let $\eta = \sum_{i=1}^r b_i \mathbf 1_{E_i}$ be a (universally-measurable) simple function (with disjoint $E_i$) such that $\eta \leq \psi^c$ and $\int \eta \, \diff\mu \geq \int \psi^c \, \diff\mu - \varepsilon/2$. By the inner regularity of $\mu$, we may find (disjoint) compact sets $K_i \subseteq E_i$ such that $\mu(E_i \setminus K_i) \leq \frac{\varepsilon}{2rb_i}$, so that, defining $g := \sum_{i=1}^r b_i \mathbf 1_{K_i},$
        \begin{equation*}
            \int g \, \diff\mu = \sum_{i=1}^r b_i \mu(K_i) \geq \sum_{i=1}^r b_i \left( \mu(E_i) - \frac{\varepsilon}{2r b_i}  \right) = \int \eta \, \diff\mu - \frac{\varepsilon}{2} \geq \int \psi^c \, \diff\mu - \varepsilon.
        \end{equation*}
        If we can show that $g^c$ is Borel-measurable, we can finish by choosing $f = -g^c$ because $-g^c(x) + g^c(y) = -g^c(x) + (-g^c)^c(y) \leq c(x,y)$ and
        \begin{align*}
            \int (-g^c) \, \diff\mu - \int (-g^c) \, \diff\nu &\geq \int g\,\diff \mu - \int \psi^c \, \diff\nu \\
            &\geq \int \psi^c \, \diff\mu - \int \psi^c \, \diff\nu - \varepsilon \geq \int \phi \, \diff\mu - \int \psi \, \diff\mu - \varepsilon
        \end{align*}
        where we have used $0 \leq g \leq - g^c \leq -\psi^{cc} = \psi^c$.

        To see that $g^c$ is Borel measurable, we again look at strict sub-level sets. And since $g \geq 0$, we know that $g^c \leq 0$, so we only need to look at $\alpha$ sub-level sets for $\alpha \leq 0$.
        \begin{align*}
            (g^c)^{-1}((-\infty,\alpha)) &= \left\{ x \in S : \inf_{y \in S} (c(x,y) - g(y)) < \alpha\right\} \\
            &= \{ x \in S : \textnormal{there exists }  y \in S \text{ such that } c(x,y) - g(y) < \alpha\} \\
            &= p(\{ (x,y) \in S \times S : c(x,y) - g(y) < \alpha\}) \\
            &= p \left( \bigcup_{n=1}^\infty \{(x,y) \in S \times S : c_n(x,y) - g(y) < \alpha\} \right) \\
            \intertext{For $y$ such that $g(y) = 0$, i.e.\ $y \notin \bigcup_{i=1}^r K_i$, the condition becomes $c_n(x,y) < \alpha \leq 0$, which is impossible because $c_n \geq 0$. So we may assume that $y \in K_1 \cup \cdots \cup K_r$. Then}
            (g^c)^{-1}((-\infty,\alpha))&= p \left( \bigcup_{n=1}^\infty \bigcup_{i=1}^r\{(x,y) \in S \times K_i : c_n(x,y) - b_i < \alpha\} \right) \\
            &= \bigcup_{n=1}^\infty \bigcup_{i=1}^r p(\{(x,y) \in S \times K_i : c_n(x,y) < \alpha + b_i \}).
        \end{align*}
        Finally, we claim that each of these projections is closed and hence Borel-measurable. Suppose $x_m \in p(\{(x,y) \in S \times K_i : c_n(x,y) < \alpha + b_i \})$ with $x_m \to x \in S$; we want to show that $x$ also lies in the projection. There exist $y_m \in K_i $ such that $c_n(x_m,y_m) < \alpha + b_i$, and by the compactness of $K_i$, by passing to a subsequence, we may assume that $y_m \to y \in K_i$. Then $y$ witnesses the membership of $x$ in the projection, as
        $$c_n(x,y) \leq \liminf_{m \to \infty} c_n(x_m,y_m) < \alpha + b_i,$$
        where we have used the lower semicontinuity of $c_n$. So these projections are closed, and we conclude that $g^c$ is Borel-measurable.
\end{proof}

\begin{remark}
    Discussions about the Borel measurability of $c$-transformed functions appear already in most standard introductions to optimal transport (e.g., \cite[Remark~5.5 and p. 69]{Villani}).
    Our result required slightly more work, since our cost function $c$ need not even be lower semi-continuous.
\end{remark}


\subsection{Existence of dual maximizer}\label{subsec:maximizer}

Our next step is to show that the Kanotorovich dual problem, i.e., the left side of \eqref{eqn:one-dual-var}, admits a maximizer for $c=1-\ind_R$.

\begin{remark}
    The analysis in this subsection does not require any hypothesis on the function $c=1-\ind_R$ beyond being uniformly bounded.
    So, the same argument applies to the two-variable appearing on the left side of \eqref{eqn:one_var_function_equality}.
\end{remark}

The idea is to take a weak limit of approximate maximizers to get a maximizer. Consider the Hilbert space $\mathcal{H}:=L^2(\rho)\times L^2(\rho)$, where $\rho := (\mu + \nu)/2$; the choice of $L^2(\rho)\times L^2(\rho)$ rather than $L^2(\mu)\times L^2(\nu)$ will be useful later. Let $\phi_n \in b\mathcal B(S)$ (with $\phi_n \ominus \phi_n \leq c$) be a sequence of $\sfrac{1}{n}$-approximate maximizers for the Kantorovich dual problem. The sequence $\{(\phi_n,-\phi_n)\}_{n\in\mathbb N}$ is norm-bounded in $\mathcal{H}$ since it is pointwise uniformly bounded, so the Banach-Alaoglu theorem guarantees that there exist some $\{n_k\}_{k\in\mathbb N}$ and $(\phi_{\infty},-\phi_{\infty})\in \mathcal{H}$ such that we have $(\phi_{n_k},-\phi_{n_k})\to (\phi_{\infty},-\phi_{\infty})$ weakly.
    
    However, notice that the functions $\phi_{\infty},-\phi_{\infty}$ are only defined up to null sets, so we cannot hope to show $\phi_{\infty} \ominus \phi_{\infty} \leq c$, which is a pointwise constraint. To obtain a pointwise cost inequality, we need to show that $\phi_{\infty}$ admits a modification $\phi'_{\infty}$ so that $(\phi'_{\infty},-\phi'_{\infty})$ together satisfy the pointwise constraint.
    The key to this is the following intermediate result, which is essentially equivalent to main result of \cite{HaydonShulman} which answers a measure-theoretic question initially posed by Arveson \cite{Arveson}; we believe it may be of interest in similar works on optimal transport in the future.

    \begin{lemma}\label{lem:Arveson}
    For $f,g:S\to \mbb R$ and $\ell:S\times S\to[0,\infty)$ any Borel-measurable functions, the following are equivalent:
    \begin{enumerate}
        \item[(a)] For all $\pi \in\Pi(\mu,\nu)$ we have $f(x)+g(y)\le \ell(x,y)$ holding for $\pi$-almost all $(x,y)\in S\times S$.
        \item[(b)] There exist Borel measurable functions $f', g':S\to \mbb R$ such that $f'(x)=f(x)$ holds for $\mu$-almost all $x\in S$, $g'(y)=g(y)$ holds $\nu$-almost all $y\in S$, and $f'(x)+g'(y)\le \ell(x,y)$ holds for all $x,y\in S$.
    \end{enumerate}
\end{lemma}

\begin{proof}
    The claim (b) implies (a) is immediate, so we only need to prove that (a) implies (b). Let $E:=\{(x,y)\in S\times S: f(x)+g(y)>\ell(x,y)\}$, and note by assumption that we have $\pi (E) = 0$ for all $\pi \in\Pi(\mu,\nu)$.
    By \cite[Corollary, p. 500]{HaydonShulman} this implies that there exist Borel sets $N_{f},N_{g}\subseteq S$ with $\mu(N_{f}) = \nu(N_{g}) = 0$ such that $E\subseteq (N_{f}\times S)\cup (S\times N_{g})$.
    Now we set
    \begin{equation*}
        f'(x) := \begin{cases}
            f(x) &\textnormal{ if } x\notin N_{f} \\
            0 &\textnormal{ if } x\in N_f \\
        \end{cases}
    \end{equation*}
    and similar for $g'$. Then $f'$ and $g'$ are as desired.
\end{proof}
    
    To apply Lemma~\ref{lem:Arveson}, we need to check that $(\phi_\infty,-\phi_\infty)$ satisfies the cost inequality $\pi$-a.s.\ for any $\pi \in \Pi(\rho,\rho)$. Define
    \begin{align*}
        h(x,y)&:= 1-\ind_{R}(x,y) - \phi_{\infty}(x) + \phi_{\infty}(y), \qquad\textnormal{ and}\\
        h_k(x,y)&:= 1-\ind_{R}(x,y) - \phi_{n_k}(x) + \phi_{n_k}(y)
    \end{align*}
    for all $k\in\mathbb N$, and observe that $\phi_{n_k}\ominus \phi_{n_k} \le c$ implies $h_{n_k} \ge 0$.
    Now take an arbitrary $\pi\in\Pi(\rho,\rho)$, and notice that that $h_{n_k}\to h$ weakly in $L^2(\pi)$.
    Thus, by the definition of weak convergence, we have:
    \begin{equation*}
        0 \le \lim_{k\to\infty}\int_{S\times S}h_k\ind_{h \le 0}\diff \pi = \int_{S\times S}h\ind_{h \le 0}\diff \pi \le 0,
    \end{equation*}
    so that $\pi(\{h < 0\}) = 0$.
    
      Lemma~\ref{lem:Arveson} now shows that we can adjust $\phi_\infty$,$-\phi_\infty$ by setting them to be 0 on $\rho$-null sets $N_\phi,N_{-\phi}$ to obtain Borel-measurable versions $\phi_\infty',(-\phi_\infty)'$ which satisfy the cost inequality pointwise; by making the modification using $N_{\phi} \cup N_{-\phi}$ instead for both functions, we can guarantee that $(-\phi_\infty)' = -\phi_\infty'$ pointwise. And since the set $N_\phi \cup N_{-\phi}$ is $\rho$-null and hence both $\mu$-null and $\nu$-null (this is where we use the choice of $L^2(\rho) \times L^2(\rho)$ rather than $L^2(\mu) \times L^2(\nu)$), we can use weak convergence to get
    \begin{align*}
        \int \phi_\infty' \, \diff\mu - \int \phi_\infty' \, \diff\nu &= \int \phi_\infty \, \diff\mu - \int \phi_\infty \, \diff\nu \\
        &= \int \phi_\infty \left( \frac{\diff \mu}{\diff \rho} - \frac{\diff \nu}{\diff \rho} \right) \, \diff\rho \\
        &= \lim_{k \to \infty} \int \phi_{n_k} \left( \frac{\diff \mu}{\diff \rho} - \frac{\diff \nu}{\diff \rho} \right) \, \diff\rho
    \end{align*}
    where we used $0 \leq \frac{\diff \mu}{\diff \rho}, \frac{\diff \nu}{\diff \rho} \leq 2$ hence $\frac{\diff \mu}{\diff \rho} - \frac{\diff \nu}{\diff \rho} \in L^2(\rho)$.
    Then
    \begin{align*}
        \int \phi_\infty' \, \diff\mu - \int \phi_\infty' \, \diff\nu  &= \lim_{k \to \infty} \int \phi_{n_k} \, \diff\mu - \int \phi_{n_k} \, \diff\nu \\
        &\geq \lim_{k \to \infty} \sup_{\substack{\phi \in b \mathcal B(S) \\ \phi \ominus \phi \leq c}} \int \phi \, \diff\mu - \int \phi \, \diff\nu - \frac{1}{n_k} \\
        &= \sup_{\substack{\phi \in b \mathcal B(S) \\ \phi \ominus \phi \leq c}} \int \phi \, \diff\mu - \int \phi \, \diff\nu.
    \end{align*}

\subsection{Conversion to indicator maximizer}\label{subsec:indicator}

As the last step of the affirmative direction, we show that the Kantorovich dual problem, i.e., the left side of \eqref{eqn:one-dual-var}, can be taken over pairs of indicator functions rather than over pairs of bounded measurable functions.

To do this, let $\phi_{\infty}':S\to\Rbb$ denote the dual maximizer from the previous subsection.
Recall from the usual ``layer cake'' decomposition that we can write
\begin{equation*}
    (\phi_\infty'(x),-\phi_\infty'(y)) = \int_0^1 (\ind_{A_t}(x),-\ind_{A_t}(y)) \, \diff t,
\end{equation*}
where we have defined $\ind_{A_t} := \{x\in S: t\le \phi_\infty'(x)\le 1\}$ for all $0\le t\le 1$.

To conclude, we claim that $A_t\in R^{\ast}$ holds for all $0\le t\le 1$. Indeed, if $(x,y)\in S\times S$ satisfies $(x,y)\notin R$ or $\phi_\infty'(x) < t$, then we of course have $\ind_{A_t}(x) - \ind_{A_t}(y) \le 1-\ind_{R}(x,y)$; otherwise, we have $(x,y)\in R$ and $t \le \phi_\infty'(x)$, so we conclude $-\phi_\infty'(y) \le - \phi_\infty'(x) \le -t$ hence $\ind_{A_t}(x) - \ind_{A_t}(y) = 1 - 1 = 0 = 1-\ind_{R}(x,y)$.

Now note that by Fubini-Tonelli, we have:
    \begin{align*}
        \int_{S}\phi_\infty'\diff \mu - \int_{S}\phi_\infty' \diff \nu &= \int_{S}\int_{0}^{1}\ind_{A_t}(x)\diff t\diff \mu(x) - \int_{S}\int_{0}^{1}\ind_{A_t}(y) \diff t\diff \nu(y) \\
        &= \int_{0}^{1}\left(\int_{S}\ind_{A_t}(x)\diff \mu(x) - \int_{S}\ind_{A_t}(y)\diff \nu(y) \right)\diff t \\
        &= \int_{0}^{1} (\mu(A_t) - \nu(A_t))\diff t \\
        &\le \sup_{A \in R^*}(\mu(A)-\nu(A)),
    \end{align*}
establishing the affirmative part of Theorem~\ref{thm:main}.

\section{Proof of Negative Part of Theorem~\ref{thm:main}}\label{sec:counter-ex}

In this section, we prove the negative part of Theorem~\ref{thm:main} (i.e., that the right side of \eqref{eqn:one-dual-var} can fail to admit minimizers) by explicitly constructing several counterexamples of interest.
We construct four counterexamples; Example~\ref{ex:counterexample} is an elementary counterexample on the space $[0,1]$, Example~\ref{ex:counterexample2} is an embedding of the first example into the pre-order of eventual domination on a suitable path space, and Example~\ref{ex:counterexample3} and Example~\ref{ex:counterexample4} are embeddings of Example~\ref{ex:counterexample} and Example~\ref{ex:counterexample2}, respectively, into path spaces in such a way that both marginal distributions satisfy the Markov property.

\begin{figure}[t]
    \centering
    \begin{tikzpicture}
          \begin{axis}[
              axis lines=middle,
              xlabel={$x$},
              ylabel={$y$},
              xmin=0, xmax=2.15,
              ymin=0, ymax=2.15,
              xtick={0,1,2},
              ytick={0,1,2},
              samples=200,
              domain=0:2,
              width=8cm,
              height=8cm,
              enlargelimits=false,
              clip=true,
              axis on top
            ]
        \addplot[domain=0:2, thick, black] {x};
        \addplot[domain=0:1, thick, dashed] {x+1};
        \addplot[
          draw=none,
          fill=black,
          fill opacity=0.1
        ] coordinates {
          (0,1)
          (0,2)
          (1,2)
        };
      \end{axis}
    \end{tikzpicture}    
    \caption{Counterexample to Strassen's theorem for $F_{\sigma}$ partial orders.}
    \label{fig:counter-ex}
\end{figure}

\begin{example}\label{ex:counterexample}
        Let $S = [0,2]$, and let $\mu$ and $\nu$ denote the Lebesgue measure on the subintervals $[0,1]$ and $[1,2]$, respectively.
        Then consider the following relation, which can also be visualized with the help of Figure~\ref{fig:counter-ex}:
        \begin{equation*}
            R := \{(x,y)\in S\times S : x=y \textnormal{ or } y > x + 1 \}.
        \end{equation*}
        Observe that $R$ is an $F_{\sigma}$ set, since we can write $R=\bigcup_{n\in\mathbb{N}}R_n$, where
        \begin{equation}\label{eqn:closed-poset}
            R_n := \left\{(x,y)\in S\times S : x=y \textnormal{ or } y \ge x + \frac{1}{n} \right\}.
        \end{equation}
        Also, $R$ is a partial order, and in particular it is reflexive and transitive.

        First, we claim that the primal value is zero.
        To see this, fix $0<\varepsilon<1$ and define the map $T_{\varepsilon}(x):[0,1]\to[1,2]$ via $T(x):=1+[x-\varepsilon]$, where $[\alpha]$ denotes the fractional part of a real number $\alpha\in\Rbb$;
    then define $\pi_{\varepsilon}\in\Pi(\mu,\nu)$ as the pushforward $\pi_{\varepsilon}:=(\textnormal{id}, T_{\varepsilon})_{\#}\mu$.
    Note $\pi_{\varepsilon}(R) = 1-\varepsilon$ for all $0 < \varepsilon < 1$, so taking $\varepsilon\to 0$ shows that the primal value is zero.

    Second, we claim that there is no coupling achieving a primal value of zero.
    To see this, assume towards a contradiction that $\pi$ is a coupling of $\mu,\nu$ with $\pi(R) = 1$.
    Then let $X$ and $Y$ be random variables defined on a common probability space, with probability $\mathbb{P}$ and expectation $\mathbb{E}$, such that the joint distribution of $(X,Y)$ is $\pi$. 
    Since the distributions $\mu,\nu$ of $X,Y$ are mutually singular, we have $\mathbb{P}(X=Y) = 0$;
    By combining this with $\mathbb{P}((X,Y)\in R) = 1$, we conclude $\mathbb{P}(Y>X+1) = 1$.
    Taking expectations yields $\mathbb{E}[Y] > \mathbb{E}[X] +1$, but this contradicts the marginal constraints which imply $\mathbb{E}[X] =\int_{0}^{1}x\diff x = \sfrac{1}{2}$ and $\mathbb{E}[Y] =\int_{1}^{2}x\diff x = \sfrac{3}{2}$.
    This shows that there is no coupling achieving the primal value.
\end{example}

Example~\ref{ex:counterexample} is counterexample to the claim ``Thorisson's working hypothesis holds for $F_{\sigma}$ partial orders,'' although only insofar as the primal problem may fail to admit minimizers.
In light of analogous work on equivalence relations \cite[Proposition~3.18]{JaffeCoupling}, it is natural to wonder whether the weaker claim ``Thorisson's working hypothesis holds for partial orders that can be written as a countable increasing union of closed partial orders'' is true; note, however, that Example~\ref{ex:counterexample} is a counterexample to this statement as well, since the closed sets in \eqref{eqn:closed-poset} are indeed partial orders.
This discrepancy highlights that there is further structure in equivalence relations, compared to in partial orders, which allows for a more complete duality theory.

Although the partial order $([0,2],R)$ from Example~\ref{ex:counterexample} may seem artificial, we note that it can be embedded into many other orders of interest, hence establishing counterexamples in other settings.
As in Section~\ref{sec:examples}, we consider \textit{eventual domination} pre-orders.

\begin{example}\label{ex:counterexample2} 
    Let $S=\Rbb^{\mathbb{N}}$, let $\preceq$ denote the pre-order of eventual domination, and let $([0,2],R)$ be the partially ordered set defined in Example~\ref{ex:counterexample}.
    Our goal is to construct a function $\phi$ such that $\phi:([0,2],R)\to(\Rbb^{\mathbb N},\preceq)$ is an order embedding, meaning we have $(x,y)\in R$ if and only if $\phi(x)\preceq\phi(y)$.
    Indeed, we may set
    \begin{equation*}
        \phi(z) := \begin{cases}
            a(z) &\textnormal{ if } 0\le z\le 1 \\
            b(z-1)
            &\textnormal{ if } 1< z\le 2 \\
        \end{cases}, \qquad\textnormal{where}
    \end{equation*}
    \begin{equation*}
        (a(t))_n := \begin{cases}
            nt+1 &\textnormal{ if } n \equiv 0 \textnormal{ mod } 3 \\
            n^2-nt &\textnormal{ if } n \equiv 1 \textnormal{ mod } 3 \\
            0 &\textnormal{ if } n \equiv 2 \textnormal{ mod } 3 \\
        \end{cases} \qquad (b(t))_n := \begin{cases}
            nt &\textnormal{ if } n \equiv 0 \textnormal{ mod } 3 \\
            n^2 &\textnormal{ if } n \equiv 1 \textnormal{ mod } 3 \\
            n^2-nt &\textnormal{ if } n \equiv 2 \textnormal{ mod } 3. \\
        \end{cases}
    \end{equation*}
    It can be checked by casework that $\phi$ is an order embedding.
    In order to transfer Example~\ref{ex:counterexample} to the present setting, we let $\mu$ and $\nu$ be the Lebesgue measures on the intervals $[0,1]$ and $[1,2]$, respectively, and then push these forward to $\tilde\mu:= \phi_{\#}\mu$ and $\tilde\nu:= \phi_{\#}\nu$ on $\Rbb^{\mathbb N}$.
    
    First, we claim that the primal problem has value zero.
    To see this, note that for any $\pi\in\Pi(\mu,\nu)$ we have
    \begin{equation*}
        \int_{S\times S}(1-\ind_{R_{\preceq}})\diff \left((\phi,\phi)_{\#}\pi\right) = \int_{[0,2]\times [0,2]}(1-\ind_R)\diff \pi.
    \end{equation*}
    Hence, an $\varepsilon$-optimal coupling for the primal problem in $([0,2],R)$ leads to an $\varepsilon$-optimal coupling for the primal problem in $(S,\preceq)$, and we conclude that the optimal value is zero.

    Last, we show that there is no primal minimizer.
    For the sake of contradiction, assume there were some $\tilde\pi\in\Pi(\tilde\mu,\tilde\nu)$ satisfying $\tilde\pi(R_{\preceq}) = 1$.
    Now note that $\phi:[0,2]\to \Rbb^{\mathbb N}$ is injective, hence it possesses a Borel left-inverse \cite[Corollary~15.2]{KechrisClassical} which we denote $\phi^{-1}:\Rbb^{\mathbb N}\to[0,2]$; then set $\pi:=(\phi^{-1},\phi^{-1})_{\#}\tilde\pi\in\Pi(\mu,\nu)$, and compute
    \begin{equation*}
        \int_{[0,2]\times [0,2]}(1-\ind_{R})\diff \pi = \int_{S\times S}(1-\ind_{R_{\preceq}})\diff \tilde \pi=0.
    \end{equation*}
    This contradicts Example~\ref{ex:counterexample} since no such $\pi$ can exist.
\end{example}

Both Example~\ref{ex:counterexample} and Example~\ref{ex:counterexample2} can be adapted into counterexamples in the setting of eventual domination of Markov chains.

\begin{example}\label{ex:counterexample3}
    The first example is just to make a constant Markov chain. Let $S,R$ be as in Example~\ref{ex:counterexample}, and let $X_0,Y_0$ be uniform on $[0,1]$,$[1,2]$, respectively. Define the constant chains $X_n = X_0$ and $Y_n = Y_0$ Then, denoting $X = (X_n)_{n=0}^\infty, Y = (Y_n)_{n=0}^\infty$, and $\preceq$ as the preorder of eventual domination in terms of $R$, we have
    $$\P(X \preceq Y) = \P((X_0,Y_0) \in R)$$
    for any coupling of $X,Y$. Therefore, as in Example~\ref{ex:counterexample}, for any $\varepsilon > 0$, there exists a coupling of $X,Y$ with $\P(X \preceq Y) \geq 1- \varepsilon$, but there exists no coupling with $\P(X \preceq Y) = 1$.
\end{example}

\begin{example}\label{ex:counterexample4}
    For a non-constant counterexample, let $S = \R^\mathbb N$ be equipped with the partial order $\leq_\text{seq}$ of domination \emph{in every coordinate}, and let $X \sim \phi_\#\mu$, $Y \sim \phi_\#\nu$, where $\mu,\nu,\phi$ are as in Example~\ref{ex:counterexample2}. Then define $S$-valued Markov chains $\widetilde X, \widetilde Y$ with path space $M = S^\mathbb N$ by successive shifts of $X$ and $Y$:
    $$\widetilde X_n = (X_k)_{k=n}^{\infty}, \qquad \widetilde Y_n = (Y_k)_{k=n}^{\infty}.$$
    Then, denoting $\preceq_{M} = \{ (\widetilde x,\widetilde y) : \widetilde x_n \preceq_{\text{seq}} \widetilde y_n \text{ for all suff.\ large $n$} \}, \preceq_S = \{ ( x, y) : x_k \leq  y_k \text{ for all suff.\ large $k$} \}$ as the preorders of eventual domination on $M$ and $S$,
    $$\P(\widetilde X \preceq_M \widetilde Y) = \P(X \preceq_S Y)$$
    for any coupling of $X,Y$ (and hence any coupling of $\widetilde X, \widetilde Y$). Therefore, as in Example~\ref{ex:counterexample2}, for any $\varepsilon > 0$, there exists a coupling of $\widetilde X,\widetilde Y$ with $\P(\widetilde X \preceq_M \widetilde Y) \geq 1- \varepsilon$, but there exists no coupling with $\P(\widetilde X \preceq_M \widetilde Y) = 1$.
\end{example}

Examples~\ref{ex:counterexample3} and \ref{ex:counterexample4} can be trivially adapted to have the Markov chains start at the same state by adding a ``root state'' that transitions to the corresponding initial states in each chain.

\bibliography{main}
\bibliographystyle{siam}

\end{document}